\documentclass[a4paper,11pt,twoside, eqno]{article}
\usepackage{mathrsfs}
\usepackage{amssymb}
\usepackage{amsmath}
\usepackage{mathrsfs}
\usepackage{amsthm}
\usepackage{amsfonts}
\usepackage{color}
\usepackage{latexsym}
\usepackage{txfonts}
\usepackage{indentfirst}
\usepackage{soul}

\usepackage{amscd}
\usepackage{bbm}
\usepackage{graphicx}
\usepackage[all]{xy}
\usepackage[mathscr]{eucal}
\usepackage{cases}
\usepackage{psfrag}
\usepackage{subfigure}

\newcommand{\N}{{\mathbb N}}

\newcommand{\R}{{\mathbb R}}

\allowdisplaybreaks
\newtheorem{theorem}{Theorem}[section]

\newtheorem{corollary}[theorem]{Corollary}

\newtheorem{remark}[theorem]{Remark}
\newtheorem{hypothesis}[theorem]{Hypothesis}

\newtheorem{claim}[theorem]{Claim}

\numberwithin{equation}{section}

\begin{document}

\title{\bf\Large A note on bifurcation theorems of Rabinowitz type
\footnotetext{\hspace{-0.35cm} 2020
{\it Mathematics Subject Classification}.
Primary: 58E07, 58E09.  Secondary: 58E05, 47J15.
\endgraf
{\it Key words and phrases.}
Rabinowitz theorem, bifurcation,  splitting theorem.
}}
\date{}
\author{Guangcun Lu\footnote{
E-mail: \texttt{gclu@bnu.edu.cn}/{February 20, 2023(first), August 24, 2023(revised)}.}}

\maketitle

\vspace{-0.7cm}

\begin{center}
\begin{minipage}{13cm}
{\small {\bf Abstract}\quad
In this note we refine the alternativity in some bifurcation theorems of Rabinowitz type,
and then improve a few of results in Lu (2022) \cite{Lu8}.
}
\end{minipage}
\end{center}
	
\vspace{0.2cm}
	
%


\section{Introduction}\label{app:App.1}


Many  bifurcation problems in infinite dimensional Banach spaces
 can be reduced to those in finite dimension spaces via various finitely dimension reductions.
For example, the proof of \cite[Theorem~0.2]{Rab77} by Rabinowitz may be reduced
to a finite-dimensional bifurcation result, which is a special case of
the following theorem.

\begin{theorem}[\hbox{\cite{IoSch, Can, CorH}}]
\label{th:A.1}
 Let $X$ be a finite dimensional normed space, let $\delta>0$, $\epsilon>0$, $\lambda^\ast\in\mathbb{R}$ and
for every $\lambda\in [\lambda^\ast-\delta, \lambda^\ast+\delta]$, let
$f_\lambda:B_X(0,\epsilon)\to\mathbb{R}$ be a function of class $C^1$.
Assume that
\begin{description}
\item[a)] the functions $\{(\lambda,u)\to f_\lambda(u)\}$ and
$\{(\lambda,u)\to f'_\lambda(u)\}$  are continuous on
$[\lambda^\ast-\delta, \lambda^\ast+\delta]\times B_X(0,\epsilon)$;
\item[b)] $u=0$ is a critical point of $f_{\lambda^\ast}$;
\item[c)] $f_\lambda$ has a
local minimum (resp. maximum) at zero for every
$\lambda\in (\lambda^\ast,\lambda^\ast+\delta]$ and
a local maximum (resp. minimum) at
zero for every $\lambda\in [\lambda^\ast-\delta, \lambda^\ast)$.
\end{description}
Then  one at least of the following assertions holds:
\begin{description}
\item[i)]  $u=0$ is not an isolated critical point of $f_{\lambda^\ast}$.
\item[ii)] $u=0$ is an isolated critical point of $f_{\lambda^\ast}$ which is a saddle point,
and for every $\lambda\ne\lambda^\ast$ in a neighborhood of $\lambda^\ast$ there is a nontrivial critical point of
$f_\lambda$ converging to zero as $\lambda\to\lambda^\ast$.
\item[iii)] $u=0$ is an isolated critical point of $f_{\lambda^\ast}$ which is a local minimizer (resp. maximizer),
and for a small one-sided neighborhood $\Lambda^0$ of $\lambda^\ast$ such that $u=0$ is a local maximizer (resp. minimizer) of $f_\lambda$ for each $\lambda\in\Lambda^0\setminus\{\lambda^\ast\}$,
$f_\lambda$ with $\lambda\in\Lambda^0\setminus\{\lambda^\ast\}$ has two distinct nontrivial critical points
converging to zero as $\lambda\to\lambda^\ast$.
\end{description}
In particular,  $(\lambda^\ast, 0)\in [\lambda^\ast-\delta, \lambda^\ast+\delta]\times B_X(0,\epsilon)$
is a bifurcation point of $f'_\lambda(u)=0$.
\end{theorem}

This result was proved by Canino \cite[Theorem~5.1]{Can} as a corollary of \cite[Theorem~2]{IoSch}
under the additional assumption that $0\in X$ as a local minimizer (or maximizer) in the assumption c)
is isolated. The above slightly weaker version is a special case of a generalization to a family of continuous functionals on
a Banach space of infinite dimension by Corvellec and Hantoute \cite[Theorem~4.2]{CorH}.

Chang \cite{Ch2} (see also \cite[Theorem II.5.5]{Ch})
gave a proof of  Morse theory  for Rabinowitz bifurcation theorem (\cite[Theorem~0.2]{Rab77}).
Motivated by the proof of \cite[Theorem~3.3]{Wa}, we can refine Chang's arguments with the method of the mountain pass theorem
to obtain the following slightly strengthened version of Theorem~\ref{th:A.1}.

\begin{theorem}\label{th:A.1.1}
Under the assumptions of Theorem~\ref{th:A.1} the conclusion iii) can be changed into:
\begin{description}
\item[iii*)] $u=0$ is an isolated critical point of $f_{\lambda^\ast}$ which is a local minimizer (resp. maximizer),
and for any given  neighborhood $W$ of $0$ in $B_X(0,\epsilon)$ there is an one-sided (right or left) neighborhood
$\Lambda^0$ of $\lambda^\ast$ such that $u=0$ is a local maximizer (resp. minimizer) of $f_\lambda$ for each $\lambda\in\Lambda^0\setminus\{\lambda^\ast\}$,
and that every $f_\lambda$, $\lambda\in\Lambda^0\setminus\{\lambda^\ast\}$,
 has  at least two  nontrivial  critical points in $W$,  saying $p_1, p_2$.
Moreover, $p_1$ and $p_2$ can be chosen
such that $f_\lambda(p_2)>f_\lambda(p_1)$ (resp. $f_\lambda(p_2)<f_\lambda(p_1)$) and $p_1$ is a strict local minimizer
(resp. maximizer) in $W$ provided that $f_\lambda$ has only finitely many  critical points in $W$ and $\dim X>1$.
\end{description}
\end{theorem}

\begin{remark}\label{rm:A.4}
{\rm  Rabinowitz's proof on the page 416 of \cite{Rab77} cannot yield two distinct nontrivial critical points
with distinct critical values in the case of $b_1=b_2=b$.
Let us see the arguments of \cite[pages 157-158]{CorH}. In the case of $c_\lambda=m_\lambda$,
$f_\lambda$ has the same value at the two critical points $u_\lambda$ and $v_\lambda$. Actually, since
$u_\lambda, v_\lambda\in{B}_{\frac{1}{k}}(0)$
are local maxima of $f_\lambda$,  using the mountain pass theorem as in our arguments below,
we may produce a new critical point $w_\lambda$ such that $f_\lambda(w_\lambda)<f_\lambda(u_\lambda)=
f_\lambda(v_\lambda)$. Finally, it seems unclear to assure that the third critical point
obtained by the proof of \cite[Theorem II.5.4]{Ch} cannot sit in $f^{-1}(c_1)$.
}
\end{remark}

In Theorem~\ref{th:A.1}, if all $f_\lambda$ are even, the case ii) cannot occur.
In the case iii), though according to the proof method of \cite[Theorem~2.2]{Rab77}
it  may be proved that $f_\lambda$  possesses  at least $\dim X$ distinct pairs of nontrivial critical points,
we cannot affirm that two distinct pairs of them have distinct critical values.

Fortunately, Wang \cite[Theorems~3.1, 3.2]{Wa} generalized the three critical point theorem
to the case of three critical orbits. Because of this we can give a more general
version of Theorem~\ref{th:A.1.1} in Section~\ref{app:App.2}, Theorem~\ref{th:A.3}.
Using the latter and Theorem~\ref{th:A.1.1} we may improve  the  celebrated  Rabinowitz bifurcation theorem \cite[Theorem~0.2]{Rab77}
and some generalizations of it, for example, those in \cite{Lu8}.
 A few of them will be given in Section~\ref{sec:3}.

\section{A finite-dimensional bifurcation theorem of Rabinowitz  type}\label{app:App.2}

The following theorem is, more or less, contained in the proofs of related results in \cite{Ch, Wa} and \cite{IoSch, CorH}.

\begin{theorem}\label{th:A.3}
Under the assumptions of Theorem~\ref{th:A.1}, let $\dim X=n$ and let an open neighborhood
$M$ of $0$ in $B_X(0,\epsilon)$ be equipped with a continuous\footnote{This action must be $C^1$ because of
a result in \cite{BoMon45, CherMar70}.} action of a compact Lie group $G$
via $C^1$-diffeomorphisms on $M$.
Suppose that $0\in{\rm Fix}_G=\{x\in M\,|\, gx=x\;\forall g\in G\}$ and that
 all $f_\lambda|_M$ are $G$-invariant.
 Then  one at least of the following assertions holds:
\begin{description}
\item[I)] $u=0$ is not an isolated critical orbit of $f_{\lambda^\ast}$.
\item[II)] $u=0$ is an isolated critical point of $f_{\lambda^\ast}$ which is a saddle point,
and for every $\lambda\ne\lambda^\ast$ in a neighborhood of $\lambda^\ast$ there is a nontrivial critical orbit of
$f_\lambda$ converging to zero as $\lambda\to\lambda^\ast$.

 \item[III)] $u=0$ is an isolated critical point of $f_{\lambda^\ast}$ which is a local minimizer (resp. maximizer),
and if $\Lambda^0$ is an one-sided neighborhood  of $\lambda^\ast$
such that $u=0$ is a local maximizer (resp. minimizer) of $f_\lambda$ for any $\lambda\in\Lambda^0\setminus\{\lambda^\ast\}$,
 then for any given small $G$-invariant neighborhood $W$ of $0$
 in $M$ there exists a $G$-invariant compact contractible neighborhood $W^\circ\subset W$ of $0$
 and an one-sided neighborhood $\Lambda^{00}\subset\Lambda^0$ of $\lambda^\ast$
  such that for every $\lambda\in\Lambda^{00}\setminus\{\lambda^\ast\}$,
 besides an orbit $\mathcal{O}_1=G(p)$ on which $f_\lambda|_{W^\circ}$ attains the minimum (resp. maximum),
(such orbits may not be unique), $f_\lambda$ has also  at least one  nontrivial critical orbit sitting in $W^\circ$
provided that the Euler-Poincar\'e characteristic of $\mathcal{O}_1$,
\begin{equation}\label{e:III}
\chi(\mathcal{O}_1)\ne 1-(-1)^n.
\end{equation}
Moreover, if $f_\lambda$, $\lambda\in\Lambda^{00}\setminus\{\lambda^\ast\}$, has only finitely many critical orbits in $W^\circ$,
and  $\mathcal{O}_1=G(p)$ is
an orbit  on which $f_\lambda|_{W^\circ}$ attains the minimum (resp. maximum),
then $f_\lambda$  has also  at least one  nontrivial critical orbit $\mathcal{O}_2$ sitting in $W^\circ$
such that $f_\lambda(\mathcal{O}_2)>f_\lambda(\mathcal{O}_1)$ (resp. $f_\lambda(\mathcal{O}_2)<f_\lambda(\mathcal{O}_1)$) provided that
 $\dim X=n>1$ and one of the following three conditions holds:
\begin{description}
\item[(III-1)]  $\dim \mathcal{O}_1 =0$.
\item[(III-2)]  $1\le\dim\mathcal{O}_1\le n-2$.
\item[(III-3)]   $1\le \dim \mathcal{O}_1=n-1$, either $\mathcal{O}_1$ is non-connected or
$\mathcal{O}_1$ is connected and $H_r(\mathcal{O}_1, \mathbb{Z}_2)\ne H_r(S^{n-1}, \mathbb{Z}_2)$ for some $0\le r\le n-1$.
\end{description}
\end{description}
\end{theorem}

When $\dim \mathcal{O}_1>0$, since $\mathcal{O}_1$ is a compact submanifold of $X$ without boundary
we have $\dim\mathcal{O}_1<n=\dim X$. In the case (III-2), $n>2$.
If the second case in (III-3) occurs we have also $n>2$ because
$H_{n-1}(\mathcal{O}_1,\mathbb{Z}_2)= \mathbb{Z}_2= H_{n-1}(S^{n-1},\mathbb{Z}_2)$
and $H_{0}(\mathcal{O}_1,\mathbb{Z}_2)= \mathbb{Z}_2= H_{0}(S^{n-1},\mathbb{Z}_2)$
imply $1\le r<n-1$.

If $G$ is a trivial group, then (III-1) is clearly satisfied, and hence
Theorem~\ref{th:A.3} includes Theorem~\ref{th:A.1.1}.
(Indeed, if $n>1$ this is true. In the case of $n=1$, the condition (2.1) is satisfied because  $1-(-1)^n=2$ and $\chi(\{p\})=1$
for any single point $p$.)

 Similarly, if the Banach space $X$ in \cite[Theorem~4.2]{CorH}
is of finite dimension,  a corresponding refinement version may be proved with the Morse theory \cite{Cor95, Ci} and
the mountain pass theorem \cite{DegM94, Kat94} for continuous functions on metric spaces.

\begin{proof}[\bf Proof of Theorem~\ref{th:A.3}]
Since $X$ has finite dimension, any two norms on it are equivalent.
Fixing an inner product on $X$ and  making an average by
means of Haar measure over $G$ we may obtain a
$G$-invariant inner product.
Therefore from now on we can assume that $B_X(0,\epsilon)$
is the ball with respect to the associated norm with the $G$-invariant inner product.
Moreover, replacing $f_\lambda$ by $f_\lambda-f_\lambda(0)$, we may assume $f_\lambda(0)=0$ for all $\lambda\in\Lambda$.

\noindent{\bf Step~1}(\textsf{Suppose that (I) does not hold}).
 Since $0\in{\rm Fix}_G$,  $u=0$ is  an isolated critical orbit (and so an isolated critical point) of $f_{\lambda^\ast}$.
By the classification of critical points of \cite[page 136]{Ja}, there exist only the following three cases:\\
{\bf Case 1}. $0\in X$ is a local minimum of $f_{\lambda^\ast}$;\\
{\bf Case 2}. $0\in X$ is a proper local maximum of $f_{\lambda^\ast}$, i.e., it is a  local maximizer of $f_{\lambda^\ast}$
and $0$ belongs to the closure of $\{f_{\lambda^\ast}<0\}$;\\
{\bf Case 3}. $0\in X$ is a saddle point of $f_{\lambda^\ast}$, i.e., $f_{\lambda^\ast}$
 takes both positive and negative values in every neighborhood of $0$.

\noindent{\bf Claim A}. {\it $u=0$ must be a strict local
minimizer (resp. maximizer) of $f_{\lambda^\ast}$ in Case 1 (resp. Case 2)}.

Indeed,  in Case 1 we may take a small neighborhood $U$ of $0$ containing $0$ as a unique critical point
of $f_{\lambda^\ast}$.
Let $\mathscr{V}_{\lambda^\ast}$ be a $C^{1-0}$ pseudo-gradient vector field of $f_{\lambda^\ast}$ on $U$.
If there exists a sequence $(z_k)\subset U\setminus\{0\}$ converging to $0$ such that
$f_{\lambda^\ast}(z_k)\equiv f_{\lambda^\ast}(0)$ for all $k$,
since $\mathscr{V}_{\lambda^\ast}(z)\ne 0$ for $z\ne U\setminus\{0\}$, moving $z_k$ along the flow of
$-\mathscr{V}_{\lambda^\ast}$ we may obtain a point $z_k'\in U$ near $z_k$
such that $f_{\lambda^\ast}(z'_k)<f_{\lambda^\ast}(z_k)=f_{\lambda^\ast}(0)$,
which contradicts the assumption that $0$ is a minimizer of $f_{\lambda^\ast}$ in $U$.
A similar proof may be completed in Case 2. Claim A is proved.

 \noindent{\bf Step~2} (\textsf{Suppose that (II) does not hold either}).  Then \\
 $\bullet$ either $0\in X$ is not a saddle point of $f_{\lambda^\ast}$ (and hence Case 3 cannot occur),\\
$\bullet$  or there exists $r\in (0,\epsilon)$ and a sequence
$\lambda_k\to \lambda^\ast$ in either $[\lambda^\ast-\delta, \lambda^\ast)$ or
 $(\lambda^\ast, \lambda^\ast+\delta]$  such that $0$ is the only
critical point of each $f_{\lambda_k}$ in $Cl(B_X(0,r))$.
Of course, by Step~1 we may also assume that $0$ is the only
critical point of $f_{\lambda^\ast}$ in $Cl(B_X(0,r))$.
Since $Cl(B_X(0,r))$ is compact and the weak slope $|df_\lambda|(u)$
of $f_\lambda$ at $u\in B_X(0,\epsilon)$ is equal to $\|f'_\lambda(u)\|$ (cf. \cite[page 1053, line 3]{Cor95}),
we may deduce that $u=0$ is  a strict local either
minimizer or maximizer of $f_{\lambda^\ast}$  as in the second paragraph of the proof of \cite[Theorem~4.2]{CorH},
and so Case 3 cannot occur again.

Therefore we must have Case 1 or Case 2, that is, $u=0$ is either a strict local
minimizer of $f_{\lambda^\ast}$ or a strict local maximizer of $f_{\lambda^\ast}$.
By considering $-f_\lambda$ we only need to study Case 1.

\noindent{\bf Step~3} (\textsf{Assume that $0$ is a strict local  minimizer  of $f_{\lambda^\ast}$}).

  \noindent{\bf Claim B}. {\it For a given small neighborhood $W$ of $0$ in $M$ there exists
  $\varepsilon>0$ such that the connected component
$W_\varepsilon$ of $\{u\in M\,|\, f_{\lambda^\ast}(u)\le\varepsilon\}$ containing $0$
is a compact neighborhood of $0$ contained in $W$ and that there are no
other nonzero critical points of $f_{\lambda^\ast}$ in $W_\varepsilon$.}

In fact, since $0$ is a strict local  minimizer  of $f_{\lambda^\ast}$ we have  $0<\epsilon'<\epsilon$ such that
$0$ is a unique minimizer  of $f_{\lambda^\ast}$ in $Cl(B_X(0,\epsilon'))$ and that
$Cl(B_X(0,\epsilon'))\subset W$ and
$\tau:=\min\{f_{\lambda^\ast}(u)\,|\,u\in \partial B_X(0,\epsilon')\}>0$.
Then $W_\varepsilon$ and $\partial B_X(0,\epsilon')$ are disjoint for any $0<\varepsilon<\tau$.
If there exists a point $v\in W_\varepsilon\setminus B_X(0,\epsilon')$,
since $W_\varepsilon$ is path connected in $M$ we have a path $\mathfrak{P}$ from $v$ to $0$ in $W_\varepsilon$.
Clearly, $\mathfrak{P}$ is interesting with $\partial B_X(0,\epsilon')$ at some $u$ and therefore we arrive at a contradiction
because $\tau\le f_{\lambda^\ast}(u)\le\varepsilon$. Hence $W_\varepsilon\subset B_X(0,\epsilon')$.
Let $(w_k)\subset W_\varepsilon$ converge to $w$. Then $f_{\lambda^\ast}(w)\le\varepsilon$,
$w\in Cl(B_X(0,\epsilon'))\subset M$ and so $w\in W_\varepsilon$. These show that
$W_\varepsilon$ is closed in $Cl(B_X(0,\epsilon'))$ and so compact in $M$ (and in $X$).
Note that $\epsilon'>0$ may be arbitrarily small and that $\tau\to 0$ as $\epsilon'\to 0$.
We conclude that $\{W_\varepsilon\,|\, 0<\varepsilon<\tau\}$ forms a neighborhood base of $0$. Claim B is proved.

Note that $0\in{\rm Fix}_G$ implies $W_\varepsilon$ to be $G$-invariant.
(Indeed, for any $g\in G$, since $g\cdot 0=0$, $g\cdot W_\varepsilon$ is also a connected compact neighborhood of $0$ contained in
$\{u\in M\,|\, f_{\lambda^\ast}(u)\le\varepsilon\}$. Then
$g\cdot W_\varepsilon\subset W_\varepsilon$. Replacing $g$ by $g^{-1}$ we get
 $g^{-1}\cdot W_\varepsilon\subset W_\varepsilon$ and so $g\cdot W_\varepsilon=W_\varepsilon$.)
Clearly, we may require that $\varepsilon$ is a regular value of $f_{\lambda^\ast}|_M$
so that\textsf{ $\partial W_\varepsilon$ is a $C^1$ submanifold.}
($W^\circ$ in III) may be chosen as $W_\varepsilon$ in the present case.)

By the assumption a) of Theorem~\ref{th:A.1}, the function $(\lambda,z)\mapsto
Df_\lambda(z)$ is continuous on
$[\lambda^\ast-\delta, \lambda^\ast+\delta]\times B_X(0,\epsilon)$.
It follows that
\begin{eqnarray*}
R_{\delta,\epsilon}:
&=&\{(\lambda, z)\in (\lambda^\ast-\delta, \lambda^\ast+\delta)\times B_X(0,\epsilon)\,|\,z\in B_X(0,\epsilon)\setminus
K(f_\lambda)\}
\end{eqnarray*}
is an open subset in $[\lambda^\ast-\delta, \lambda^\ast+\delta]\times B_X(0,\epsilon)$, where
$K(f_\lambda)$ denotes the critical set of $f_\lambda$. By \cite[Lemma~5.2]{Lu8}
there exists a $C^\infty$ map $R_{\delta,\epsilon}\to X,\;(\lambda,z)\mapsto\mathscr{V}_\lambda(z)$,
such that for each $\lambda\in(\lambda^\ast-\delta, \lambda^\ast+ \delta)$ the map
$\mathscr{V}_\lambda: B_X(0,\epsilon)\setminus
K(f_\lambda)\to X$  satisfies
\begin{equation}\label{e:A.1}
\|\mathscr{V}_\lambda(z)\|\le 2\|Df_{\lambda}(z)\|\quad\hbox{and}\quad
\langle Df_{\lambda}(z), \mathscr{V}_\lambda(z)\rangle\ge
\|Df_{\lambda}(z)\|^2
\end{equation}
for all $z\in B_X(0,\epsilon)\setminus K(f_\lambda)$, i.e., $\mathscr{V}_\lambda$ is a pseudo-gradient vector field of
$f_\lambda$ in Palais' sense. Denote by $\varphi_\lambda^t$ the flow of $-\mathscr{V}_\lambda$.
Then $W_\varepsilon$ is invariant under $\varphi^t_{\lambda^\ast}$, and so contractible. The latter claim
leads to $\chi(W_\varepsilon)=1$. (Here $\chi(W_\varepsilon)$ denotes the Euler-Poincar\'e characteristic of $W_\varepsilon$.)
   Since $\inf\{\|Df_{\lambda^\ast}(z)\|\,|\,z\in\partial W_\varepsilon\}>0$
implies $\inf\{\|\mathscr{V}_{\lambda^\ast}(z)\|\,|\,z\in\partial W_\varepsilon\}>0$ and
the map $R_{\delta,\epsilon}\ni (\lambda,z)\mapsto\mathscr{V}_\lambda(z)\in X$ is continuous,
we have $0<\delta_0<\delta$ such that
\begin{equation}\label{e:A.2}
\inf\{\langle Df_{\lambda^\ast}(z), \mathscr{V}_\lambda(z)\rangle\,|\,z\in\partial W_\varepsilon,\,\lambda^\ast-\delta_0\le\lambda\le\lambda^\ast+\delta_0\}>0.
\end{equation}
It follows that \textsf{all $\varphi^t_{\lambda}$ with $|\lambda-\lambda^\ast|\le\delta_0$ preserve $W_\varepsilon$.}

\noindent{\bf Step~4} (\textsf{The case  that $0$ is a local maximizer of $f_\lambda$ for every
$\lambda\in (\lambda^\ast,\lambda^\ast+\delta]$}). The same method deals with the case that
$0$ is a local maximizer of $f_\lambda$ for every $\lambda\in [\lambda^\ast-\delta, \lambda^\ast)$.
Let us fix a $\lambda\in (\lambda^\ast,\lambda^\ast+\delta_0]$ below.

If $f_\lambda|_{W_\varepsilon}$ has infinitely many critical orbits, we are done.

\textsf{From now on we suppose that
$f_\lambda|_{W_\varepsilon}$ has only finitely many critical orbits},
$$
\mathcal{O}_0=\{0\}, \mathcal{O}_1,\cdots,\mathcal{O}_m,
$$
where  $\mathcal{O}_1=G(p)$ is an orbit on which $f_\lambda|_{W_\varepsilon}$ attains the minimum.
(Of course, other orbits $\mathcal{O}_i$ might have this property.)
Let $C_\ast(f_\lambda, \mathcal{O}_i;\mathbb{Z}_2)$ be
 the critical group of $f_\lambda$ at $\mathcal{O}_i$ with coefficients in $\mathbb{Z}_2$ (cf. \cite{Wa} for definition).
Then
 \begin{equation}\label{e:A.3}
C_{k}(f_\lambda,0;\mathbb{Z}_2)=\delta_{kn}\mathbb{Z}_2\quad\hbox{and}\quad C_{k}(f_\lambda,\mathcal{O}_1;\mathbb{Z}_2)=H_k(\mathcal{O}_1,\mathbb{Z}_2)\;\forall k
\end{equation}
 by Example~1 on page 33 of \cite{Ch} and \cite[Example~2.1]{Wa}, respectively.
 Let
 $$
 c_\lambda=\max\{f_\lambda|_{\mathcal{O}_i}\,|\,0\le i\le m\}.
 $$
 Since (\ref{e:A.1}) and (\ref{e:A.2}) imply that $\mathcal{O}_i\in W_\varepsilon\setminus\partial W_\varepsilon$, $i=0, 1,\cdots,m$,
 we conclude that
  \begin{equation*}
 c_\lambda<b_\lambda:=\min\{f_\lambda(u)\,|\, u\in \partial W_\varepsilon\}\quad\forall
 \lambda\in[\lambda^\ast-\delta_0, \lambda^\ast+\delta_0]
 \end{equation*}
  if $\delta_0>0$ is small enough.
 Otherwise, we have a sequence $\lambda_k\to\lambda^\ast$ and $x^k\in{\rm Crit}(f_{\lambda_k}|_{W_\varepsilon})$
 such that $f_{\lambda_k}(x^k)\ge b_{\lambda_k}$ for $k=1,2,\cdots$.
Let $u_k\in \partial W_\varepsilon$ such that $f_{\lambda_k}(u_k)=b_{\lambda_k}$ for $k=1,2,\cdots$.
 Since $W_\varepsilon$ is compact, we may assume $x^k\to x^0$ and $u_k\to u_0\in \partial W_\varepsilon$.
 It follows from the assumption a) of Theorem~\ref{th:A.1} that $f'_{\lambda^\ast}(x^0)=0$
 and $f_{\lambda^\ast}(x^0)\ge f_{\lambda^\ast}(u_0)\ge b_{\lambda^\ast}=\varepsilon$, which is a contradiction.
(Take $\Lambda^{00}$ in III) as $[\lambda^\ast,\lambda^\ast+\delta_0]$ in the present case.)

 Take  $d_\lambda\in (c_\lambda, b_\lambda)$.
Since $d_\lambda$ is a regular value of $f_\lambda|_{W_\varepsilon}$ and $W_\varepsilon\setminus\{u\in W_\varepsilon\,|\,f_\lambda(u)< d_\lambda\}$
contains no critical points of $f_\lambda|_{W_\varepsilon}$,
we can use $\varphi^t_{\lambda}$ to construct a strong deformation retract from
$W_\varepsilon$ to $(f|_{W_\varepsilon})_{d_\lambda}:=\{f|_{W_\varepsilon}\le d_\lambda\}$ and hence
\begin{equation}\label{e:contrac}
\hbox{$(f|_{W_\varepsilon})_{d_\lambda}$ is contractible},
\end{equation}
which implies its Euler-Poincar\'e characteristic
$\chi((f|_{W_\varepsilon})_{d_\lambda})=1$.
By \cite[Theorem I.7.2]{Ch} we see that \cite[Theorems~1.1, 1.2]{Wa1} are applicable to $((f|_{W_\varepsilon})_{d_\lambda}, \emptyset)$.
Then we get
$$
\sum^m_{i=0}\sum^\infty_{k=0}(-1)^k{\rm rank}C_k(f_\lambda, \mathcal{O}_i;\mathbb{Z}_2)=\chi((f|_{W_\varepsilon})_{d_\lambda})=1
$$
and so
\begin{equation}\label{e:A.4}
\sum^m_{i=1}\sum^\infty_{k=0}(-1)^k{\rm rank}C_k(f_\lambda, \mathcal{O}_i;\mathbb{Z}_2)=1-(-1)^n
\end{equation}
by (\ref{e:A.3}).

   Suppose that $f_\lambda|_{W_\varepsilon}$ has only critical orbits $\mathcal{O}_0, \mathcal{O}_1$, i.e., $m=1$. Then
   (\ref{e:A.4}) and  the second equality in (\ref{e:A.3}) lead to
   $$
 \chi(\mathcal{O}_1)= \sum^\infty_{k=0}(-1)^k{\rm rank}H_k(\mathcal{O}_1;\mathbb{Z}_2)=1-(-1)^n,
$$
which contradicts (\ref{e:III}). The claims before ``Moreover'' in (III) are completed.

Assuming $\dim X=n>1$, let us prove the part after ``Moreover'' in (III).

\noindent{\bf Proof under the condition (III-1)}. Since $\dim\mathcal{O}_1=0$ and
 $\mathcal{O}_1$ is a compact manifold, $\mathcal{O}_1$ must consist of finitely many points,
  saying $p_1=p,\cdots,p_\ell$. Let $p_0=0$.

  \underline{Suppose that $G$ is nontrivial and $p\notin{\rm Fix}_G$}.
  Then there exists $g\in G$ such that $g\cdot p\ne p$, and so $\ell\ge 2$.
   By the mountain pass lemma (cf. \cite[Proposition~5.42]{MoMoPa})
we have $p^\ast\in{\rm Crit}(f_\lambda|_{W_\varepsilon})$ such that
$$
f_\lambda(p_1)=f_\lambda(p_2)<f_\lambda(p^\ast)=\inf_{\gamma\in\Gamma}\max_{[0,1]}(f_\lambda\circ\gamma),
$$
where $\Gamma:=\{\gamma\in C([0,1], W_\varepsilon)\,|\, \gamma(0)=p_1,\,\gamma(1)=p_2\}$.
(Though $W_\varepsilon$ is not a Banach space, since it is invariant for $\varphi^t_{\lambda}$
we can still use \cite[Proposition~5.42]{MoMoPa}, see \cite{Rab77} and explanations of \cite[page 294]{Ja}.)
By \cite[Theorem]{Hofer85} this $p^\ast$ can be required to be
 either a local minimum or of mountain-pass type. The latter case means that
  the set $\{y\in U\,|\, f_\lambda(y)<f_\lambda(p^\ast)\}$ is nonempty and not path-connected
  for any open neighborhood $U$ of $p^\ast$.
  Note that strict local maxima on  spaces of more than one dimension
  cannot be mountain pass points  (cf. \cite[Proposition~5.1]{Ja}).
  Therefore $p^\ast\ne 0$ because $\dim X=n>1$ and $p_0=0$ is a strict local maximizer of $f_\lambda$ by the above assumption.
Then  $\mathcal{O}_1=G(p)$ and $\mathcal{O}_2=G(p^\ast)$ are desired critical orbits of $f_\lambda$.

\underline{Suppose $p\in{\rm Fix}_G$}. (This includes the case that $G$ is trivial.) Then $\ell=1$ and $\mathcal{O}_1=\{p\}$.
If $f_\lambda|_{W_\varepsilon}$ has only critical orbits $\mathcal{O}_0=\{0\}$ and $\mathcal{O}_1$,
since $C_{k}(f_\lambda,p;\mathbb{Z}_2)=\delta_{k0}\mathbb{Z}_2$ by Example~1 on page 33 of \cite{Ch},
it follows from (\ref{e:A.4})  that $1+(-1)^n=1$, which is
a contradiction.
Therefore $f_\lambda|_{W_\varepsilon}$ has the third critical orbit
$\mathcal{O}^\ast$. Suppose that  there exists a point $q$ in $\mathcal{O}^\ast$ such that it is not a local minimizer of $f_\lambda|_{W_\varepsilon}$.
Then $f_\lambda|_{\mathcal{O}^\ast}>f_\lambda|_{\mathcal{O}_1}$, and we are done.
Otherwise, $\mathcal{O}^\ast$ has a small compact neighborhood $\Omega$ disjoint with $\mathcal{O}_1$, such that
$f_\lambda|_{\mathcal{O}^\ast}<\inf\{f_\lambda(x)\,|\,x\in\partial\Omega\}$.
Fixing any $q\in \mathcal{O}^\ast$, as above we may use the mountain pass lemma to yield
a $p^\ast\in{\rm Crit}(f_\lambda|_{W_\varepsilon})$ such that
$$
f_\lambda(p_1)\le f_\lambda(q)<f_\lambda(p^\ast)=\inf_{\gamma\in\Gamma}\max_{[0,1]}(f_\lambda\circ\gamma),
$$
where $\Gamma:=\{\gamma\in C([0,1], W_\varepsilon)\,|\, \gamma(0)=p_1,\,\gamma(1)=q\}$, and that
$p^\ast$ is either a local minimizer or of mountain-pass type.
Using the assumption $\dim X=n>1$ the same reason as above also leads to $p^\ast\ne 0$.
Hence  $\mathcal{O}_1=G(p)$ and $\mathcal{O}_2=G(p^\ast)$ are desired critical orbits of $f_\lambda$.

\noindent{\bf Proofs under the condition (III-2) or (III-3)}.
Suppose that $f_\lambda|_{W_\varepsilon}$ has only critical orbits $\mathcal{O}_0, \mathcal{O}_1$.
Since $f_\lambda|_{\mathcal{O}_1}<f_\lambda|_{\mathcal{O}_0}$,
by \cite[Theorem I.7.2]{Ch} we may use (the proof of) \cite[Lemma~3.1(2)]{Wa} to get
\begin{equation}\label{e:A.5}
C_k(f_\lambda, 0;\mathbb{Z}_2)=C_k(f_\lambda,\mathcal{O}_0;\mathbb{Z}_2)=H_{k-1}(\mathcal{O}_1,\mathbb{Z}_2)\quad\forall k\ge 2.
\end{equation}
Indeed, in the present case we have $c_\lambda=f_\lambda|_{\mathcal{O}_0}=0$ by the assumption above Step 1. Let $a_\lambda=f_\lambda|_{\mathcal{O}_1}$.
Take $\rho>0$ so small that $a_\lambda+\rho<c_\lambda$. Consider the long exact sequence of
a tripe
$$
\left(({f_\lambda}|_{W_\varepsilon})_{d_\lambda}, ({f_\lambda}|_{W_\varepsilon})_{c_\lambda-\rho}, ({f_\lambda}|_{W_\varepsilon})_{a_\lambda-\rho} \right)=
\left(({f_\lambda}|_{W_\varepsilon})_{d_\lambda}, ({f_\lambda}|_{W_\varepsilon})_{c_\lambda-\rho}, \emptyset\right):
$$
\begin{eqnarray}\label{e:A.6}
&&\cdots\to H_k(({f_\lambda}|_{W_\varepsilon})_{c_\lambda-\rho};\mathbb{Z}_2)\to
H_k(({f_\lambda}|_{W_\varepsilon})_{d_\lambda}; \mathbb{Z}_2)\nonumber\\
&&\to H_k(({f_\lambda}|_{W_\varepsilon})_{d_\lambda}, ({f_\lambda}|_{W_\varepsilon})_{c_\lambda-\rho} ;\mathbb{Z}_2)\to H_{k-1}(({f_\lambda}|_{W_\varepsilon})_{c_\lambda-\rho};\mathbb{Z}_2)\to \cdots
\end{eqnarray}
By \cite[Theorem I.7.2]{Ch} we may use  \cite[Theorem~2.1]{Wa} and (\ref{e:A.3}) to derive
\begin{eqnarray*}
&&H_k(({f_\lambda}|_{W_\varepsilon})_{c_\lambda-\rho};\mathbb{Z}_2)=C_k(f_\lambda,\mathcal{O}_1;\mathbb{Z}_2)=H_k(\mathcal{O}_1,\mathbb{Z}_2),\\
&&H_k(({f_\lambda}|_{W_\varepsilon})_{d_\lambda}; \mathbb{Z}_2)=\delta_{k0}\mathbb{Z}_2,\\
&&H_k(({f_\lambda}|_{W_\varepsilon})_{d_\lambda}, ({f_\lambda}|_{W_\varepsilon})_{c_\lambda-\rho} ;\mathbb{Z}_2)=C_k(f_\lambda,\mathcal{O}_0;\mathbb{Z}_2)=\delta_{kn}\mathbb{Z}_2,
\end{eqnarray*}
where the second equality is because of (\ref{e:contrac}).
These and (\ref{e:A.6}) immediately leads to (\ref{e:A.5}).

\underline{For case (III-2)},   (\ref{e:A.5}) and the first equality in (\ref{e:A.3}) lead to
$$
\mathbb{Z}_2=C_{n}(f_\lambda,0;\mathbb{Z}_2)=H_{n-1}(\mathcal{O}_1,\mathbb{Z}_2)=0
$$
since  $1\le \dim\mathcal{O}_1\le n-2$. This is a  contradiction.

\underline{For case (III-3)}, note that $n\ge 2$ and $H_{n-1}(\mathcal{O}_1,\mathbb{Z}_2)=(\mathbb{Z}_2)^s$
 (cf. Exercise~4.8 on the page 213 of \cite{Mass}), where $s$ is the number of component components of $\mathcal{O}_1$.

If $s>1$, it follows from (\ref{e:A.5}) and the first equality in (\ref{e:A.3}) that
$$
\mathbb{Z}_2=C_{n}(f_\lambda,0;\mathbb{Z}_2)=H_{n-1}(\mathcal{O}_1,\mathbb{Z}_2)= (\mathbb{Z}_2)^s,
$$
which is a  contradiction.

If $s=1$, that is, $\mathcal{O}_1$ is a connected and compact manifold of dimension $n-1$,
then $H_{n-1}(\mathcal{O}_1,\mathbb{Z}_2)= \mathbb{Z}_2= H_{n-1}(S^{n-1},\mathbb{Z}_2)$
and $H_{0}(\mathcal{O}_1,\mathbb{Z}_2)= \mathbb{Z}_2= H_{0}(S^{n-1},\mathbb{Z}_2)$. Hence $1\le r<n-1$
and $H_{r}(\mathcal{O}_1,\mathbb{Z}_2)\ne H_{r}(S^{n-1},\mathbb{Z}_2)=0$.
By (\ref{e:A.5}) and the first equality in (\ref{e:A.3}) we derive
\begin{equation*}
0=C_{r+1}(f_\lambda, 0;\mathbb{Z}_2)=C_{r+1}(f_\lambda,\mathcal{O}_0;\mathbb{Z}_2)=H_{r}(\mathcal{O}_1,\mathbb{Z}_2)\ne 0,
\end{equation*}
and hence a contradiction.

In summary, besides the orbits $\mathcal{O}_0$ and $\mathcal{O}_1$,
$f_\lambda$ has also the third critical orbit  $\mathcal{O}^\ast$ sitting in $W_\varepsilon$.
If $f_\lambda|_{\mathcal{O}^\ast}>f_\lambda|_{\mathcal{O}_1}$, we are done.
Otherwise, $f_\lambda|_{\mathcal{O}^\ast}=f_\lambda|_{\mathcal{O}_1}$.
As in the arguments under the case of ``$p\in{\rm Fix}_G$'' above
we may obtain the desired critical orbit $\mathcal{O}_2$ of $f_\lambda$.

The proof of the theorem is completed.
\end{proof}

There exists a closely related result to Theorem~\ref{th:A.3},
 \cite[Theorem~5.1]{Lu8}, which  may be restated as the following more precise version.

\begin{theorem}[\hbox{\cite[Theorem~5.1]{Lu8}}]\label{th:A.2}
Under the assumptions of Theorem~\ref{th:A.1},
let $X$ be equipped with a continuous\footnote{Actually,
we only need to assume  ``$\pi$ is continuous with respect to $g\in G$''.
Indeed, by theorems in \cite{BoMon45, CherMar70}  this assumption implies $\pi$ to be $C^\infty$
since all $\pi_g$ belong to $\mathscr{L}(X)$ and hence $C^\infty$.}
 action $\pi$ of a compact Lie group $G$ via linear isometries so that
each $f_\lambda$ is invariant under the $G$-action.
Suppose also
\begin{description}
\item[(I)]  the local minimums (resp. maximums) at zero in assumption c) of Theorem~\ref{th:A.1} are strict,
\item[(II)] $u=0$ is an isolated critical point of $f_{\lambda^\ast}$.
 (This is possible because $0\in{\rm Fix}_G$.)
 \end{description}
Then when the Lie group  $G$ is equal to $\mathbb{Z}_2=\{{\rm id}_X, -{\rm id}_X\}$ (resp. $S^1$
without fixed points except $0$, which implies $\dim X$ to be an even more than one),
 for a given neighborhood $\mathcal{B}$ of $0$ in $B_X(0,\epsilon)$ one at least of the following assertions holds:
\begin{description}
\item[(i)]  $u=0$ is a local minimizer  of $f_{\lambda^\ast}$,  and for
 a small one-sided  neighborhood $\Lambda^0$ of $\lambda^\ast$
such that $u=0$ is a local maximizer of $f_\lambda$ for each
$\lambda\in\Lambda^0\setminus\{\lambda^\ast\}$, and $G=\{{\rm id}_X, -{\rm id}_X\}$ (resp. $S^1$),
$f_\lambda$ with $\lambda\in\Lambda^0\setminus\{\lambda^\ast\}$
  has either infinitely many distinct $G$-orbits  in $\mathcal{B}$
with critical values uniformly less than $f_\lambda(0)$
  or at least $\dim X$ (resp. $\frac{1}{2}\dim X$) distinct  nontrivial critical $G$-orbits  in $\mathcal{B}$
   with distinct critical values
  less than $f_\lambda(0)$ provided that $G=\{{\rm id}_X, -{\rm id}_X\}$ (resp. $S^1$).

\item[(ii)] $u=0$ is a local maximizer  of $f_{\lambda^\ast}$,  and for
 a small one-sided neighborhood $\Lambda^0$ of $\lambda^\ast$ such that
 $u=0$ is a local  minimizer  of $f_\lambda$ for each
$\lambda\in\Lambda^0\setminus\{\lambda^\ast\}$, and $G=\{{\rm id}_X, -{\rm id}_X\}$ (resp. $S^1$), $f_\lambda$ with
$\lambda\in\Lambda^0\setminus\{\lambda^\ast\}$ has either infinitely many distinct $G$-orbits
 in $\mathcal{B}$ with critical values uniformly greater than $f_\lambda(0)$
or at least $\dim X$ (resp. $\frac{1}{2}\dim X$) distinct  nontrivial critical
$G$-orbits  in $\mathcal{B}$  with distinct critical values greater than $f_\lambda(0)$.

\item[(iii)] $u=0$ is a saddle point of $f_{\lambda^\ast}$, and for a small neighborhood
$\Lambda^+$ (resp. $\Lambda^-$) of $\lambda^\ast$ such that $u=0$ is a local maximizer (resp.  minimizer) of $f_\lambda$
for each $\lambda$ in $\Lambda^+\setminus\{\lambda^\ast\}$ (resp. $\Lambda^-\setminus\{\lambda^\ast\}$),
and $G=\{{\rm id}_X, -{\rm id}_X\}$ or $S^1$, $f_\lambda$ has
either infinitely many distinct $G$-orbits  in $\mathcal{B}$
with critical values uniformly less (resp. greater) than $f_\lambda(0)$ or
at least $n^+$ (resp. $n^-$) distinct  nontrivial critical
$G$-orbits  in $\mathcal{B}$  with distinct critical values less (resp. greater) than $f_\lambda(0)$;
moreover $n^++n^-\ge\dim X$ (resp. $\frac{1}{2}\dim X$) if $G=\{{\rm id}_X, -{\rm id}_X\}$ (resp. $S^1$).
\end{description}
\end{theorem}
\begin{proof}[\bf Proof]
Indeed, by the assumption a) of Theorem~\ref{th:A.1}, replacing
$f_\lambda$ by $f_\lambda-f_\lambda(0)$  we may assume $f_\lambda(0)=0$
for all $\lambda$. The assumption (I) implies that one of the following holds:
\begin{description}
\item[(A)]  $0\in X$ is a strict local minimizer (resp. maximizer) of $f_{\lambda}$
for all $\lambda\in [\lambda^\ast-\delta, \lambda^\ast)$ (resp.
$\lambda\in (\lambda^\ast, \lambda^\ast+\delta]$).
\item[(B)]  $0\in X$ is a strict local maximizer (resp. minimizer) of $f_{\lambda}$
for all $\lambda\in [\lambda^\ast-\delta, \lambda^\ast)$ (resp.
$\lambda\in (\lambda^\ast, \lambda^\ast+\delta]$).
\end{description}

Since $u=0$ is an isolated critical point of $f_{\lambda^\ast}$ by (II),  the classification of critical points of \cite[page 136]{Ja}
implies that there exist mutually disjoint:\\
\noindent{\bf Case 1}. $0\in X$ is a local minimizer of $f_{\lambda^\ast}$;\\
\noindent{\bf Case 2}. $0\in X$ is a proper local maximizer of $f_{\lambda^\ast}$;\\
\noindent{\bf Case 3}. $0\in X$ is a saddle point of $f_{\lambda^\ast}$.

Let $S^+$ and $S^-$ be as in \cite[page 1278]{Lu8}. Then $S^+\ne\emptyset$ and $S^-=\emptyset$ in Case 1,
$S^+=\emptyset$ and $S^-\ne\emptyset$ in Case 2, and $S^+\ne\emptyset$ and $S^-\ne\emptyset$ in Case 3.
Let $T^+$ and $T^-$ be as in \cite[Lemma~5.4]{Lu8}. For $\delta>0$ small enough, and any
 $\lambda\in [\lambda^\ast-\delta, \lambda^\ast+\delta]$, by \cite[Lemma~5.4]{Lu8} we have:\\
\noindent{\underline{In Case 1}}, $\min\{f_\lambda(z)\,|\,z\in T^+\}>0$ and
\begin{eqnarray*}
&&i_{\mathbb{Z}_2}(T^+)\ge\dim X\quad\hbox{if}\quad G=\mathbb{Z}_2=\{{\rm id}_X, -{\rm id}_X\},\\
&&i_{S^1}(T^+)\ge\frac{1}{2}\dim X \quad\hbox{if}\quad G=S^1.
\end{eqnarray*}
\noindent{\underline{In Case 2}}, $\max\{f_\lambda(z)\,|\,z\in T^-\}<0$ and
\begin{eqnarray*}
&&i_{\mathbb{Z}_2}(T^-)\ge\dim X\quad\hbox{if}\quad G=\mathbb{Z}_2=\{{\rm id}_X, -{\rm id}_X\},\\
&&i_{S^1}(T^-)\ge\frac{1}{2}\dim X \quad\hbox{if}\quad G=S^1.
\end{eqnarray*}
\noindent{\underline{In Case 3}},  $\min\{f_\lambda(z)\,|\,z\in T^+\}>0$,
$\max\{f_\lambda(z)\,|\,z\in T^-\}<0$ and
\begin{eqnarray*}
&&i_{\mathbb{Z}_2}(T^+)+ i_{\mathbb{Z}_2}(T^-)\ge\dim X\quad\hbox{if}\quad G=\mathbb{Z}_2=\{{\rm id}_X, -{\rm id}_X\},\\
&&i_{S^1}(T^+)+ i_{S^1}(T^-)\ge\frac{1}{2}\dim X \quad\hbox{if}\quad G=S^1.
\end{eqnarray*}

First, we assume that (A) holds. By the arguments on the page 1280 of \cite{Lu8} we see that
Claims~5.6,5.7 in \cite{Lu8} may be restated as the following precise versions.

\begin{claim}\label{cl:F4.5}
Given a neighborhood $\mathcal{B}$ of $0$ in $B_X(0,\epsilon)$
there exists a $G$-invariant compact subset $T^-$ in $\mathcal{B}$ and $\delta_0\in (0,\delta]$ such that
if $i_{G}(T^-)=k>0$ and $\lambda\in [\lambda^\ast-\delta_0, \lambda^\ast)$ then
$f_{\lambda}$ has either infinitely many distinct $G$-orbits in $\mathcal{B}$
 with critical values uniformly greater than $f_\lambda(0)$
or at least $k$ distinct  nontrivial critical $G$-orbits in $\mathcal{B}$
 with distinct critical values greater than $f_\lambda(0)$, which also converge to $0$ as $\lambda\to\lambda^\ast$.
\end{claim}

\begin{claim}\label{cl:F4.6}
Given a neighborhood $\mathcal{B}$ of $0$ in $B_X(0,\epsilon)$
there exists a $G$-invariant compact subset $T^+$ in $\mathcal{B}$ and $\delta'_0\in (0,\delta]$ such that
if $i_{G}(T^+)=l>0$ and $\lambda\in (\lambda^\ast, \lambda^\ast+\delta_0']$ then
$f_{\lambda}$ has either infinitely many distinct $G$-orbits in $\mathcal{B}$
with critical values uniformly less than $f_\lambda(0)$ or at least $l$ distinct  nontrivial
critical $G$-orbits in $\mathcal{B}$ with distinct critical values  less than $f_\lambda(0)$,
 which also converge to $0$ as $\lambda\to\lambda^\ast$.
\end{claim}

For Case 2, take $\Lambda^0=[\lambda^\ast-\delta_0, \lambda^\ast]$.
By Claim~\ref{cl:F4.5} we obtain the desired conclusion.

For Case 1, by considering $-f_\lambda$ we should choose $\Lambda^0=[\lambda^\ast, \lambda^\ast+\delta_0']$
and get the desired conclusion by Claim~\ref{cl:F4.6}.

For Case 3, we deduce that $f_\lambda$ has at least $n^+=i_{G}(T^+)$ (resp. $n^-=i_{G}(T^-)$) distinct nontrivial
critical $G$-orbits by Claim~\ref{cl:F4.6} (resp. Claim~\ref{cl:F4.5}).
Therefore $\Lambda^+$ and $\Lambda^-$ should be $[\lambda^\ast, \lambda^\ast+\delta_0']$
and $[\lambda^\ast-\delta_0, \lambda^\ast]$, respectively.

Next, if (B) holds then $\Lambda^0=[\lambda^\ast, \lambda^\ast+\delta_0']$ in Case 1, $\Lambda^0=[\lambda^\ast-\delta_0, \lambda^\ast]$
in Case 2, and $\Lambda^+=[\lambda^\ast-\delta_0, \lambda^\ast]$ and $\Lambda^-=[\lambda^\ast, \lambda^\ast+\delta_0']$.
\end{proof}

\begin{remark}\label{rm:twoTh}
{\rm Theorem~\ref{th:A.3} and Theorem~\ref{th:A.2} cannot be included each other.
Both are complementary. The assumptions in Theorem~\ref{th:A.2} are stronger. }
\end{remark}

\section{Infinite-dimensional bifurcation theorems of Rabinowitz  or Fadell-Rabinowitz type}\label{sec:3}

In this section we first prove Theorem~\ref{th:A.7}, an improvement of
the parameterized splitting theorem \cite[Theorem~A.3]{Lu8} under slightly weaker assumptions,
and a bifurcation theorem (Theorem~\ref{th:A.9}) as a consequence.
From Theorem~\ref{th:A.7} and Theorem~\ref{th:A.1.1} (resp. \cite[Theorem~5.11]{Lu8}) we directly obtain
improvements of \cite[Theorem~4.6]{Lu8} (resp. \cite[Theorem~5.12]{Lu8}) and
Theorem~\ref{th:A.10} (resp. Theorem~\ref{th:A.11}).
Using Theorems~\ref{th:A.3},~\ref{th:A.7}
we may prove an equivariant bifurcation theorem (Theorem~\ref{th:A.13}), which generalizes
Theorem~\ref{th:A.10} (and \cite[Theorem~4.6]{Lu8}).
By Theorem~\ref{th:A.3} we also prove a generalization of \cite[Theorem~4.2]{Lu8},
Theorem~\ref{th:A.15}. Finally, we give improvements of
\cite[Theorems~5.18,~5.19]{Lu8}, Theorems~\ref{th:Bi.3.20.1},~\ref{th:Bi.3.20.2},
respectively.

\begin{hypothesis}[{\cite[Hypothesis~1.1]{Lu8}}]\label{hyp:A.4}
{\rm Let $H$ be a Hilbert space with inner product $(\cdot,\cdot)_H$
and the induced norm $\|\cdot\|$, and let $X$ be a dense linear subspace in $H$.
Let  $U$ be an open neighborhood of $0$ in $H$,
and let $\mathcal{L}\in C^1(U,\mathbb{R})$ satisfy $\mathcal{L}'(0)=0$.
Assume that the gradient $\nabla\mathcal{L}$ has a G\^ateaux derivative $B(u)\in \mathscr{L}_s(H)$ at every point
$u\in U\cap X$, and that the map $B: U\cap X\to
\mathscr{L}_s(H)$  has a decomposition
$B=P+Q$, where for each $x\in U\cap X$,  $P(x)\in\mathscr{L}_s(H)$ is  positive definitive and
$Q(x)\in\mathscr{L}_s(H)$ is compact. Maps $B$, $P$ and $Q$ are also assumed to satisfy the following
properties:
\begin{description}
\item[(D1)]  $\{u\in H\,|\, B(0)u=\mu u,\;\mu\le 0\}\subset X$.
\item[(D2)] For any sequence $(x_k)\subset
U\cap X$ with $\|x_k\|\to 0$, it holds that $\|P(x_k)u-P(0)u\|\to 0$ for any $u\in H$.
\item[(D3)] The  map $Q: U\cap X\to \mathscr{L}_s(H)$ is continuous at $0$ with respect to the topology
on $H$.
\item[(D4)] For any sequence $(x_k)\subset U\cap X$ with $\|x_k\|\to 0$, there exist
 constants $C_0>0$ and $k_0\in\N$ such that
$(P(x_k)u, u)_H\ge C_0\|u\|^2$ for all $u\in H$ and for all $k\ge k_0$.
\end{description}}
\end{hypothesis}

The condition ({\rm D4}) is equivalent to the following
\begin{description}
\item[\bf (D4*)] There exist positive constants $\eta_0>0$ and  $C'_0>0$ such that $\bar{B}_H(0,\eta_0)\subset U$ and
$$
(P(x)u, u)\ge C'_0\|u\|^2\quad\forall u\in H,\;\forall x\in
\bar{B}_H(0,\eta_0)\cap X.
$$
\end{description}

\begin{hypothesis}[{\cite[Hypothesis~1.3]{Lu8}}]\label{hyp:A.5}
{\rm Let $H$ be a Hilbert space with inner product $(\cdot,\cdot)_H$
and the induced norm $\|\cdot\|$, and let $X$ be a Banach space with
norm $\|\cdot\|_X$, such that $X\subset H$ is dense in $H$ and $\|x\|\le\|x\|_X\;\forall x\in X$.
For an open neighborhood $U$ of $0$ in $H$, $U\cap X$
is also an open neighborhood of $0$ in $X$, denoted by $U^X$.
 Let $\mathcal{L}:U\to\mathbb{R}$ be a continuous functional  satisfying  the following
conditions:
\begin{description}
\item[(F1)] $\mathcal{L}$ is continuously directional
differentiable and $D\mathcal{L}(0)=0$.
\item[(F2)] There exists a continuous and continuously directional differentiable
 map $A: U^X\to X$, which is also  strictly Fr\'{e}chet differentiable
 at $0$,   such that
$D\mathcal{ L}(x)[u]=(A(x), u)_H$ for all $x\in U\cap X$ and $u\in X$.
\item[(F3)] There exists a map $B: U\cap X\to \mathscr{L}_s(H)$ such that
$(DA(x)[u], v)_H=(B(x)u, v)_H$ for all $x\in U\cap X$  and
$u, v\in X$. (So $B(x)$ induces an element in $\mathscr{L}(X)$, denoted by $B(x)|_X$,
and $B(x)|_X=DA(x)\in\mathscr{L}(X),\;\forall x\in U\cap X$.)
\item[(C)]   $\{u\in H\,|\, B(0)(u)\in X\}\subset X$,
in particular  ${\rm Ker}(B(0))\subset X$.
\item[(D)] $B$ satisfies the same conditions as in Hypothesis~\ref{hyp:A.4}.
\end{description}}
\end{hypothesis}

\noindent{\bf 3.1}. {\bf A slight improvement of \cite[Theorem~A.3]{Lu8}
and a sufficient criterion for bifurcations.}
The following is only the parameterized splitting theorem \cite[Theorem~A.3]{Lu8} under the weaker action conditions of groups.

\begin{theorem}\label{th:A.7}
Let $H$, $X$ and $U$ be as in Hypothesis~\ref{hyp:A.5},
 and $\Lambda$ a topological space.
Let $\mathcal{L}_\lambda\in C^1(U, \mathbb{R})$, $\lambda\in\Lambda$, be a continuous family of functionals
    satisfying $\mathcal{L}'_\lambda(0)=0$ for all $\lambda\in\Lambda$.
 For each $\lambda\in\Lambda$, assume that
there exist maps $A_\lambda\in C^1(U^X, X)$ and $B_\lambda:U\cap X\to\mathscr{L}_s(H)$
such that: {\rm a)} $\Lambda\times U^X\ni (\lambda, x)\to A_\lambda(x)\in X$ is continuous;  {\rm b)}
 \begin{equation}\label{e:LAB}
 D\mathcal{L}_\lambda(x)[u]=(A_\lambda(x), u)_H\quad\hbox{and}\quad
(DA_\lambda(x)[u], v)_H=(B_\lambda(x)u, v)_H
\end{equation}
 for all $x\in U\cap X$  and $u, v\in X$; {\rm c)}  $B_\lambda$ has a decomposition
$B_\lambda=P_\lambda+Q_\lambda$, where for each $x\in U\cap X$,
 $P_\lambda(x)\in\mathscr{L}_s(H)$ is  positive definitive and
$Q_\lambda(x)\in\mathscr{L}_s(H)$ is compact.
   Let $0\in H$ be a degenerate critical point of
  some $\mathcal{L}_{\lambda^\ast}$, i.e.,  ${\rm Ker}(B_{\lambda^\ast}(0))\ne\{0\}$.
Suppose also that $P_\lambda$ and $Q_\lambda$ satisfy the following conditions:
    \begin{description}
\item[(i)]  For each $h\in H$, it holds that $\|P_{\lambda}(x)h-P_{\lambda^\ast}(0)h\|\to 0$
as $x\in U\cap X$ approaches to $0$ in $H$ and $\lambda\in\Lambda$ converges to $\lambda^\ast$.

 \item[(ii)]  For some small $\delta>0$, there exists a positive constant $c_0>0$ such that
$$
(P_\lambda(x)u, u)\ge c_0\|u\|^2\quad\forall u\in H,\;\forall x\in
\bar{B}_H(0,\delta)\cap X,\quad\forall\lambda\in \Lambda.
$$
 \item[(iii)]  $Q_\lambda: U\cap X\to \mathscr{L}_s(H)$ is uniformly continuous at $0$  with respect to $\lambda\in \Lambda$.
  \item[(iv)]  If $\lambda\in \Lambda$ converges to $\lambda^\ast$ then
  $\|Q_{\lambda}(0)-Q_{\lambda^\ast}(0)\|\to 0$.
   \item[(v)] $(\mathcal{L}_{\lambda^\ast}, H, X, U, A_{\lambda^\ast}, B_{\lambda^\ast}=P_{\lambda^\ast}+ Q_{\lambda^\ast})$ satisfies Hypothesis~\ref{hyp:A.5}.
    \end{description}
 Let $H^+_\lambda$, $H^-_\lambda$ and $H^0_\lambda$ be the positive definite, negative definite and zero spaces of
${B}_\lambda(0)$.  Denote by $P^0_\lambda$ and $P^\pm_\lambda$ the orthogonal projections onto
 $H^0_\lambda$ and $H^\pm_\lambda=H^+_\lambda\oplus H^-_\lambda$,
and by $X^\star_\lambda=X\cap H^\star_\lambda$ for $\star=+,-$, and by  $X^\pm_\lambda=P^\pm_\lambda(X)$.
 Then there exists a neighborhood $\Lambda_0$ of $\lambda^\ast$ in $\Lambda$,
$\epsilon>0$, a (unique) $C^0$ map
\begin{equation}\label{e:Spli.2.1.1}
\psi:\Lambda_0\times B_{H^0_{\lambda^\ast}}(0,\epsilon)\to X^\pm_{\lambda^\ast}
\end{equation}
which is $C^1$ in the second variable and
satisfies $\psi(\lambda, 0)=0$ for all $\lambda\in \Lambda_0$ and
\begin{equation}\label{e:Spli.2.1.2}
 P^\pm_{\lambda^\ast}A_\lambda(z+ \psi(\lambda,z))=0\quad\forall (\lambda,z)\in \Lambda_0
 \times B_{H^0_{\lambda^\ast}}(0,\epsilon),
 \end{equation}
an open neighborhood $\mathcal{W}$ of $\Lambda_0\times\{0\}$ in $\Lambda_0\times H$ and a homeomorphism
\begin{eqnarray}\label{e:Spli.2.1.3}
&&\Lambda_0\times B_{H^0_{\lambda^\ast}}(0,\epsilon)\times
\left(B_{H^+_{\lambda^\ast}}(0, \epsilon) + B_{H^-_{\lambda^\ast}}(0, \epsilon)\right)\to  \mathcal{W},\nonumber\\
&&\hspace{20mm}({\lambda}, z, u^++u^-)\mapsto ({\lambda},\Phi_{{\lambda}}(z, u^++u^-))
\end{eqnarray}
satisfying $\Phi_{{\lambda}}(0)=0$, such that for each $\lambda\in \Lambda_0$,
$\Phi_{{\lambda}}$ is a homeomorphism  from
$B_{H^0_{\lambda^\ast}}(0,\epsilon)\oplus
B_{H^+_{\lambda^\ast}}(0, \epsilon)\oplus B_{H^-_{\lambda^\ast}}(0, \epsilon)$
onto an open neighborhood $\mathcal{W}_\lambda:=\{v\in H\,|\, (\lambda,v)\in \mathcal{W}\}$ of $0$ in $H$, and
\begin{eqnarray}\label{e:Spli.2.2}
&&\mathcal{L}_{\lambda}\circ\Phi_{\lambda}(z, u^++ u^-)=\|u^+\|^2-\|u^-\|^2+ \mathcal{
L}_{{\lambda}}(z+ \psi({\lambda}, z))\\
&& \quad\quad \forall (z, u^+ + u^-)\in  B_{H^0_{\lambda^\ast}}(0,\epsilon)\times
\left(B_{H^+_{\lambda^\ast}}(0, \epsilon) + B_{H^-_{\lambda^\ast}}(0, \epsilon)\right).\nonumber
\end{eqnarray}
 Moreover, there also hold: {\bf (A)}
 $$
d_z\psi(\lambda,z)=-[P^\pm_{\lambda^\ast}\circ({B}_{\lambda}(z+\psi(\lambda,z))|_{X^\pm_{\lambda^\ast}})]^{-1}
\circ(P^\pm_{\lambda^\ast}\circ({B}_\lambda(z+\psi(\lambda,z))|_{H^0_{\lambda^\ast}})).
$$

\noindent{\bf (B)} The functional
\begin{equation}\label{e:Spli.2.3}
\mathcal{L}_{\lambda}^\circ: B_{H^0_{\lambda^\ast}}(0,\epsilon)\to \mathbb{R},\;
z\mapsto\mathcal{L}_{\lambda}(z+ \psi({\lambda}, z))
\end{equation}
 is of class $C^{2}$, its first-order and second-order differentials  at $z\in
B_{H^0}(0, \epsilon)$ are given by
 \begin{eqnarray}\label{e:Spli.2.4}
&& d\mathcal{L}^\circ_\lambda(z)[\zeta]=\bigl(A_\lambda(z+ \psi(\lambda, z)), \zeta\bigr)_H\quad\forall \zeta\in H^0,\\
  &&d^2\mathcal{L}^\circ_\lambda(0)[z,z']=\left(P^0_{\lambda^\ast}\bigr[{B}_\lambda(0)-
 {B}_\lambda(0)(P^\pm_{\lambda^\ast}{B}_{\lambda}(0)|_{X^\pm_{\lambda^\ast}})^{-1}
 (P^\pm_{\lambda^\ast}{B}_\lambda(0))\bigr]z, z'\right)_H,\nonumber\\
&& \hspace{40mm} \forall z,z'\in H^0.\label{e:Spli.2.5}
 \end{eqnarray}

\noindent{\bf (C)}  Suppose that $\pi:G\times H\to H, (g,u)\mapsto \pi_gu$
is a continuous action of a topological group $G$ via linear isometries on $H$,
and that each $\pi_g$ also restricts to a linear isometry from $(X,\|\cdot\|_X)$ to itself.
If both $U$ and $\mathcal{L}_\lambda$ are $G$-invariant, then $H^0_\lambda$, $H^\pm_\lambda$
are $G$-invariant subspaces, and for each $\lambda\in \Lambda$,
the above maps $\psi(\lambda, \cdot)$  and $\Phi_{\lambda}(\cdot,\cdot)$  are
  $G$-equivariant, and $\mathcal{L}^\circ_{\lambda}$ is $G$-invariant.
  [If $G$ is a Lie group, by a result in \cite{BoMon45} the induced $G$-action on $H^0_\lambda$
  is $C^\infty$ because $\dim H^0_\lambda<\infty$. When $G$ is a compact Lie group, then the assumption
``each $\pi_g$ also restricts to a linear isometry from $(X,\|\cdot\|_X)$ to itself''
may be replaced by ``$G\ni g\mapsto g\cdot x\in X$ is continuous for any $x\in X$,
(therefore $G\times X\ni (g,x)\mapsto g\cdot x\in X$ is also continuous by \cite[Theorem~1]{CherMar70}),
and each $\pi_g$ is  a linear continuous map from $(X,\|\cdot\|_X)$ to itself''.]

\noindent{\bf (D)} If for some $p\in\mathbb{N}$,  $\Lambda$ is a $C^p$ manifold and
$\Lambda\times U^X\ni (\lambda,x)\mapsto A(\lambda, x):=A_\lambda(x)\in X$ is $C^p$, then so is $\psi$.
\end{theorem}

\begin{remark}\label{rm:A.8}
{\rm {\bf (i)} The difference between Theorem~\ref{th:A.7} and
\cite[Theorem~A.3]{Lu8} is that the first sentence in (C) is replaced by
``If a compact Lie group $G$  acts on $H$ orthogonally, which induces  $C^1$ isometric actions on $X$''.

\noindent{\bf (ii)} If the topological group $G$ in (C) is a Baire space,
since $H\ni u\mapsto\pi_gu\in H$ is continuous,
by \cite[Theorem~1]{CherMar70} the map $\pi$ is continuous if and only if
$G\ni g\mapsto \pi_gu\in H$ is continuous for each $u\in H$.
Note that $\pi$ in (C) is not required  to induce an action on $X$, i.e.,
$G\times X\ni (g,x)\mapsto g\cdot x\in X$ is continuous, or equivalently
$G\ni g\mapsto g\cdot x\in X$ is continuous for any $x\in X$ (by \cite[Theorem~1]{CherMar70}).

\noindent{\bf (iii)} In order to prove the second claim in the bracket of (C),
 we only need to construct an equivalent norm $\|\cdot\|_X^\ast$ to $\|\cdot\|_X$ such that
\begin{eqnarray}\label{e:EquivNorm}
\|h\cdot x\|_X^\ast=\|x\|_X\quad\hbox{and}\quad\|x\|\le\|x\|_X^\ast,\quad\forall (h, x)\in G\times X.
\end{eqnarray}
To this goal let us fix a right invariant Haar measure $\mu$ on $G$. Since $G\ni g\mapsto \|g\cdot x\|_X\in\mathbb{R}$ is continuous,
$$
\|x\|_X^\ast:=\frac{1}{|G|}\int_G\|g\cdot x\|_X\mu(dg)
$$
is well-defined, where $|G|$ is the volume of $G$ with respect to $\mu$.
Note that $\|\cdot\|_X^\ast$ is $G$-invariant. This can be seen from the following calculation:
\begin{eqnarray*}
\|h\cdot x\|_X^\ast&=&\frac{1}{|G|}\int_G\|g\cdot (h\cdot x)\|_X\mu(dg)\\
&=&\frac{1}{|G|}\int_G\|(gh)\cdot x\|_X\mu(dg)=\frac{1}{|G|}\int_G\|g\cdot x\|_X\mu(dg)=\|x\|_X^\ast
\end{eqnarray*}
by the change of variable $g\mapsto gh^{-1}$. Since $\|g\cdot x\|=\|x\|\le\|x\|_X$ for $(g,x)\in G\times X$, we deduce
\begin{eqnarray}\label{e:compareNorm}
\|x\|=\frac{1}{|G|}\int_G\|g\cdot x\|\mu(dg)\le \frac{1}{|G|}\int_G\|g\cdot x\|_X\mu(dg)=\|x\|_X^\ast,\quad\forall x\in X.
\end{eqnarray}
We also need to prove that the norms $\|\cdot\|_X^\ast$ and $\|\cdot\|_X$ are equivalent.
Since $G$ is compact and $G\ni g\mapsto \|g\cdot x\|_X\in\mathbb{R}$ is continuous, we have $\sup_{g\in G}\|\pi_gx\|_X<+\infty$
and hence $\sup_{g\in G}\|\pi_g\|_{\mathscr{L}(X)}<+\infty$ by the uniformly bounded principle.
It follows that $\|x\|^\ast_X\le\sup_{g\in G}\|\pi_g\|_{\mathscr{L}(X)}\|x\|_X$ for any $x\in X$.
On the other hand $\|x\|_X=\|\pi_g(\pi_g)^{-1}x\|_X\le\|\pi_g\|_{\mathscr{L}(X)}\|(\pi_g)^{-1}x\|_X$ and so
$$
\|(\pi_g)^{-1}x\|_X\ge \frac{1}{\sup_{g\in G}\|\pi_g\|_{\mathscr{L}(X)}}\|x\|_X.
$$
But $(\pi_g)^{-1}=\pi_{g^{-1}}$. We derive that
$$
\|x\|_X^\ast=\frac{1}{|G|}\int_G\|g\cdot x\|_X\mu(dg)\ge \frac{1}{\sup_{g\in G}\|\pi_g\|_{\mathscr{L}(X)}}\|x\|_X.
$$
Hence the norms $\|\cdot\|_X^\ast$ and $\|\cdot\|_X$ are equivalent.

\noindent{\bf (iv)} If the condition ``${\rm Ker}(B_{\lambda^\ast}(0))\ne\{0\}$'' in Theorem~\ref{th:A.7} is changed into
``${\rm Ker}(B_{\lambda^\ast}(0))=\{0\}$'', a shorter proof gives rise to a generalization of
\cite[Theorem A.2]{Lu8}:
\textsf{Then there exists a neighborhood $\Lambda_0$ of $\lambda^\ast$ in $\Lambda$, $\epsilon>0$, a family of open neighborhoods of $0$ in
$H$, $\{W_\lambda\,|\, \lambda\in\Lambda_0\}$,
and a family of origin-preserving homeomorphisms,
$$
\phi_\lambda: B_{H^+_{\lambda^\ast}}(0,\epsilon) +
B_{H^-_{\lambda^\ast}}(0,\epsilon)\to W_\lambda, \quad \lambda\in\Lambda_0,
$$
 such that
$$
\mathcal{L}_{\lambda}\circ\phi_\lambda(u^++ u^-)=\|u^+\|^2-\|u^-\|^2,
\quad\forall (u^+, u^-)\in B_{H^+_{\lambda^\ast}}(0,\epsilon)\times
B_{H^-_{\lambda^\ast}}(0,\epsilon).
$$
Moreover, $\Lambda_0\times (B_{H^+_{\lambda^\ast}}(0,\epsilon) +
B_{H^-_{\lambda^\ast}}(0,\epsilon))\ni (\lambda, u)\mapsto \phi_\lambda(u)\in H$
is continuous, and $0$ is an isolated critical point of each $\mathcal{L}_{\lambda}$ with $\lambda\in\Lambda_0$.}
}
\end{remark}

\begin{proof}[\bf Proof of Theorem~\ref{th:A.7}]
(A) and (B) were proved in \cite{Lu8}. (D) is clear by their proof.
We here give  detailed proofs for the conclusions in (C) because
 the present conditions are weaker than those in \cite[Theorem~A.3(iii)]{Lu8}.

Take $\eta>0$ so small that
$B_{H^0_{\lambda^\ast}}(0,\eta)\oplus B_{H^\pm_{\lambda^\ast}}(0,\eta)\subset U$ and hence
$B_{H^0_{\lambda^\ast}}(0,\eta)\oplus B_{X^\pm_{\lambda^\ast}}(0,\eta)\subset U^X$.
Follow the notations in the proof of \cite[Theorem A.3]{Lu8}.
Let
$$
\pi:G\times H\to H,\,(g,x)\mapsto g\cdot x=\pi_gx
$$
 be the given $G$-action. That is, $\pi$ is continuous and
$\pi_g:H\to H$ is linear and satisfies $(\pi_gx, \pi_gy)_H=(x,y)_H$ for all $x,y\in H$.
Since $\mathcal{L}_\lambda$ is $G$-invariant we have
$D\mathcal{L}_\lambda(g\cdot x)[\pi_gu]=D\mathcal{L}_\lambda(x)[u]$ for any $(x,u)\in U\times H$, and so
\begin{equation}\label{e:LAB-1}
(A_\lambda(g\cdot x), \pi_gu)_H=(A_\lambda(x), u)_H=(\pi_gA_\lambda(x), \pi_gu)_H\quad\forall (x,u)\in U^X\times H
\end{equation}
by (\ref{e:LAB}), which implies
\begin{equation}\label{e:LAB-2}
A_\lambda(g\cdot x)=\pi_gA_\lambda(x)\quad\forall x\in U^X.
\end{equation}
 Since $A_\lambda\in C^1(U^X, X)$, we derive from (\ref{e:LAB-1}) that
$$
(DA_\lambda(g\cdot x)[\pi_gv], \pi_gu)_H=(DA_\lambda(x)[v], u)_H= (\pi_gDA_\lambda(x)[v], \pi_gu)_H \quad\forall v\in X.
$$
This and the second equality in (\ref{e:LAB}) lead to
$$
(B_\lambda(g\cdot x)\pi_gv, \pi_gu)_H=(B_\lambda(x)[v], u)_H= (\pi_gB_\lambda(x)v, \pi_gu)_H \quad\forall v\in X.
$$
Since $\pi_g\in\mathscr{L}(H)$ and $X$ is dense in $H$, this implies
\begin{equation}\label{e:LAB-3}
B_\lambda(g\cdot x)\pi_g=\pi_gB_\lambda(x)\quad\forall (g, x)\in G\times U^X.
\end{equation}
It follows  that $H^0_{\lambda}$ and $H^\pm_{\lambda}$ are invariant subspaces for $\pi_g$,
and $\pi_gP^0_{\lambda}=P^0_{\lambda}\pi_g$ and $\pi_gP^\pm_{\lambda}=P^\pm_{\lambda}\pi_g$.
Since $\pi_g:H\to H$ is a linear isometry, for any $r>0$ it holds that
\begin{equation}\label{e:LAB-3.1}
\pi_g\big(B_{H^\star_{\lambda^\ast}}(0,r)\big)=B_{H^\star_{\lambda^\ast}}(0,r),\quad\star=+,-,0.
\end{equation}
Moreover, each $\pi_g$ also restricts to a linear isometry from $(X,\|\cdot\|_X)$ to itself. Therefore
\begin{equation}\label{e:LAB-3.2}
\pi_g(X^\star)=X^\star\quad\hbox{and}\quad
\pi_g\left(B_{X^\star_{\lambda^\ast}}(0,r)\right)=B_{H^\star_{\lambda^\ast}}(0,r),\quad\star=+,-,0,
\end{equation}
where $X^\star=X\cap H^\star$, $\star=+,-,0$.
By these, (\ref{e:LAB-2}) implies that  the map
$$
\mathscr{A}:\Lambda\times B_{H^0_{\lambda^\ast}}(0,\eta)\oplus B_{X^\pm_{\lambda^\ast}}(0,\eta)\to
X^\pm_{\lambda^\ast},\;(\lambda,z,x)\mapsto P^\pm_{\lambda^\ast}(A_\lambda(z+x))
$$
satisfies $\mathscr{A}(\lambda, g\cdot z, g\cdot x)=\pi_g\mathscr{A}(\lambda, z, x)$
for any $(g, z, x)\in G\times B_{H^0_{\lambda^\ast}}(0,\eta)\oplus B_{X^\pm_{\lambda^\ast}}(0,\eta)$.
In \cite{Lu8} we had used the implicit function theorem to get a  compact neighborhood $\Lambda_0$ of $\lambda^\ast$ in $\Lambda$,
$2\delta\in (0, \eta)$,
and a unique $C^0$ map $\psi:\Lambda_0\times B_{H^0_{\lambda^\ast}}(0,2\delta)\to  B_{X^\pm_{\lambda^\ast}}(0,\eta)$
such that
$$
\mathscr{A}(\lambda, z,\psi(\lambda,z))=P^\pm_{\lambda^\ast}(A_\lambda(z+\psi(\lambda,z)))\equiv 0,\quad\forall (\lambda,z)\in
\Lambda\times B_{H^0_{\lambda^\ast}}(0,2\delta).
$$
In particular, $P^\pm_{\lambda^\ast}(A_\lambda(g\cdot z+\psi(\lambda,g\cdot z)))\equiv 0\;\forall g\in G$. Moreover, by
(\ref{e:LAB-2}) we have also
\begin{eqnarray*}
0&=&\pi_gP^\pm_{\lambda^\ast}(A_\lambda(z+\psi(\lambda,z)))\\
&=&P^\pm_{\lambda^\ast}\pi_g(A_\lambda(z+\psi(\lambda,z)))=P^\pm_{\lambda^\ast}(A_\lambda(g\cdot z+g\cdot\psi(\lambda,z)))
\end{eqnarray*}
for any $(\lambda,z)\in\Lambda\times B_{H^0_{\lambda^\ast}}(0,2\delta)$. By the assumptions,
$\pi_g(X)\subset X$ and $\pi_g\in\mathscr{L}(X)$. Then
$$
\Lambda_0\times B_{H^0_{\lambda^\ast}}(0, 2\delta)\ni (\lambda,z)\mapsto g\cdot\psi(\lambda,z)=
\pi_g\psi(\lambda,z)\in  B_{X^\pm_{\lambda^\ast}}(0,\eta)
$$
is continuous. Hence the uniqueness of $\psi$ leads to
\begin{equation}\label{e:LAB-4}
\pi_g\psi(\lambda,z)=\psi(\lambda, g\cdot z),\quad\forall (\lambda,z)\in\Lambda\times B_{H^0_{\lambda^\ast}}(0, 2\delta).
\end{equation}

By the proof of \cite[Theorem~A.3]{Lu8},  shrinking $\Lambda_0$ and $\delta>0$
(if necessary) we can obtain positive constants $\mathfrak{a}_1$ and $\mathfrak{a}_2$ such that
for each $\lambda\in\Lambda_0$ the map
\begin{eqnarray*}
{\bf F}_\lambda: B_{H^0_{\lambda^\ast}}(0,2\delta)\oplus B_{H^\pm_{\lambda^\ast}}(0,2\delta)\to \mathbb{R}
\end{eqnarray*}
given by ${\bf F}_\lambda(z,u)=\mathcal{L}_{\lambda}(z+\psi({\lambda}, z)+u)- \mathcal{L}_{{\lambda}}(z+ \psi({\lambda}, z))$
satisfies
\begin{eqnarray}\label{e:LAB-5}
 &(D_2{\bf F}_{{\lambda}}(z, u^+ + u^-_2)-D_2{\bf F}_{{\lambda}}(z, u^++ u^-_1))[u^-_2-u^-_1]\le
-\mathfrak{a}_1\|u^-_2-u^-_1\|^2,\\
&D_2{\bf F}_{{\lambda}}(z, u^++u^-)[u^+-u^-]\ge  \mathfrak{a}_2(\|u^+\|^2+ \|u^-\|^2)\label{e:LAB-6}
\end{eqnarray}
for all $\lambda\in \Lambda_0$,
$z\in B_{H^0_{\lambda^\ast}}(0,2\delta)$ and $u^+\in B_{H^+_{\lambda^\ast}}(0,2\delta)$, $u^-\in B_{H^-_{\lambda^\ast}}(0,2\delta)$.

Consider the topological normed vector bundle
$$
p:\mathcal{E}=\Lambda_0\times \bar{B}_{H^0_{\lambda^\ast}}(0,\delta)\oplus H^\pm\to
\Lambda_0\times \bar{B}_{H^0_{\lambda^\ast}}(0,\delta),\;(\lambda,z, u)\to (\lambda, z).
$$
It has a natural a Finsler structure
$|||\cdot |||:\Lambda_0\times \bar{B}_{H^0_{\lambda^\ast}}(0,\delta)\oplus H^\pm\to\mathbb{R}$
given by
\begin{eqnarray}\label{e:LAB-7}
|||(\lambda, z,u)|||:=\|u\|_H.
\end{eqnarray}
Then $\mathcal{E}=\mathcal{E}^+\oplus\mathcal{E}^-$, where $\mathcal{E}^\ast=\Lambda_0\times \bar{B}_{H^0_{\lambda^\ast}}(0,\delta)\oplus H^\ast$,
$\ast=+,-$, and
$$
B_{2\delta}(\mathcal{E})=\{(\lambda,z, u)\in\mathcal{E}\;|\;|||(\lambda, z,u)|||<2\delta\}=
\Lambda_0\times \bar{B}_{H^0_{\lambda^\ast}}(0,\delta)\oplus B_{H^\pm_{\lambda^\ast}}(0,2\delta).
$$
Define $J:B_{2\delta}(\mathcal{E})\to\mathbb{R}$ by $J(\lambda, z,u):={\bf F}_\lambda(z,u)$.
The restriction of it to the fiber $B_{2\delta}(\mathcal{E})_{(\lambda,z)}\equiv B_{H^\pm_{\lambda^\ast}}(0,2\delta)$
is given by $J_{(\lambda,z)}(u)={\bf F}_\lambda(z,u)$ and so $DJ_{(\lambda,z)}(u)=D_2{\bf F}_\lambda(z,u)$. Then (\ref{e:LAB-5}) and (\ref{e:LAB-6}) imply:
\begin{description}
\item[(i)] $J_{(\lambda,z)}(0)=0$ and $DJ_{(\lambda,z)}(0)=0$.
\item[(ii)] $(DJ_{(\lambda,z)}(u+v_2)-DJ_{(\lambda,z)}(u+v_1))[v_2-v_1]>0$
for $(\lambda,z,u)\in\bar{B}_{\delta}(\mathcal{E}^+)$ and $y_i\in (\bar{B}_{\delta}(\mathcal{E}^-))_{(\lambda,z)}$,
$i=1,2$, $v_1\ne v_2$.
\item[(iii)] $DJ_{(\lambda,z)}(u+v)[u-v]>0$ for any $u,v\in (\bar{B}_{\delta}(\mathcal{E}^+))_{(\lambda,z)}$ with $(u,v)\ne (0,0)$.
\item[(iv)] $DJ_{(\lambda,z)}(u)[u]=D_2{\bf F}_{{\lambda}}(z, u)[u]\ge  \mathfrak{a}_2\|u\|_H^2>p(\|u\|_H)$
for $(\lambda,z,u)\in\bar{B}_{\delta}(\mathcal{E}^+)$, where $p(t)=\frac{1}{2}t^2$ for $t\ge 0$.
\end{description}

Define actions of $G$ on $\Lambda_0\times \bar{B}_{H^0_{\lambda^\ast}}(0,\delta)$ and $\mathcal{E}$ by
$$
g\cdot(\lambda,z)=(\lambda, g\cdot z)\quad\hbox{and}\quad
g\cdot(\lambda,z, u)=(\lambda, g\cdot z, g\cdot u).
$$
Then $p(g\cdot(\lambda,z, u))=(\lambda, g\cdot z)=g\cdot(p(\lambda,z,u))$ and
$$
\mathcal{E}_{(\lambda,z)}\equiv H^\pm\ni u\mapsto g\cdot u\in H^\pm\equiv\mathcal{E}_{g\cdot(\lambda,z)}
$$
 is a vector space isomorphism. Clearly, $g\cdot\mathcal{E}^\ast\subset\mathcal{E}^\ast$, $\ast=+,-$, and
  \begin{eqnarray*}
 |||g(\lambda, z,u)|||_{g\cdot(\lambda,z)}=\|g\cdot u\|_H=\|u\|_H=|||(\lambda, z,u)|||_{(\lambda,z)}.
  \end{eqnarray*}
Moreover,  (\ref{e:LAB-4}) implies
$$
J(g\cdot(\lambda, z,u)):={\bf F}_\lambda(g\cdot z,g\cdot u)={\bf F}_\lambda(z,u)=J(\lambda,z,u).
$$
Because of these and (i)-(iv) above, applying  \cite[Theorem~A.2]{Lu3} to $J$ we get for some small $\epsilon_1\in (0, \delta)$
a preserving-fiber homeomorphism from
$$
B_{\epsilon_1/2}(\mathcal{E}^+)\oplus B_{\epsilon_1/2}(\mathcal{E}^-)=\Lambda_0\times \bar{B}_{H^0_{\lambda^\ast}}(0,\delta)\oplus B_{H^+_{\lambda^\ast}}(0,\epsilon_1/2)\oplus
B_{H^-_{\lambda^\ast}}(0,\epsilon_1/2)
$$
to a $G$-invariant open neighborhood $\widehat{\mathcal{W}}$ of the zero section of
$\mathcal{E}$,
$$
(\lambda,z,u)\mapsto \Phi(\lambda,z,u)=(\lambda,z, \Phi_{(\lambda,z)}(u)),
$$
such that $\Phi_{(\lambda,z)}(0)=0$ and
\begin{eqnarray*}
&&{\bf F}_\lambda(z, \Phi_{(\lambda,z)}(u))=J(\Phi(\lambda, z,u))=\|u^+\|^2_H-\|u^-\|^2_H,\\
&&\Phi(g\cdot(\lambda,z,u))=g\cdot\Phi(\lambda,z,u),\quad i.e., \quad
\Phi_{(\lambda, g\cdot z)}(g\cdot u)=g\cdot\Phi_{(\lambda,z)}(u).
\end{eqnarray*}
Let
$$
\widehat{\mathcal{W}}_{(\lambda,z)}=\{u\in B_{H^+_{\lambda^\ast}}(0,\epsilon_1/2)\oplus
B_{H^-_{\lambda^\ast}}(0,\epsilon_1/2)\,|\, (\lambda,z,u)\in \widehat{\mathcal{W}}\}.
$$
It is an open neighborhood of $0$ in $H^\pm_{\lambda^\ast}$, and
$\Phi_{(\lambda,z)}$ is a homeomorphism from
$B_{H^+_{\lambda^\ast}}(0,\epsilon_1/2)\oplus B_{H^-_{\lambda^\ast}}(0,\epsilon_1/2)$ onto $\widehat{\mathcal{W}}_{(\lambda,z)}$.
Write $\Phi_\lambda(z,u)=(z, \Phi_{(\lambda,z)}(u))$. Then
\begin{eqnarray*}
\Phi_\lambda(g\cdot(z,u))&=&
\Phi_\lambda(g\cdot z, g\cdot u)\\
&=&(g\cdot z, \Phi_{(\lambda,g\cdot z)}(g\cdot u))=(g\cdot z, g\cdot\Phi_{(\lambda,z)}(u))=g\cdot\Phi_\lambda(z,u).
\end{eqnarray*}
for any $g\in G$ and $(\lambda, z, u)\in
\Lambda_0\times \bar{B}_{H^0_{\lambda^\ast}}(0,\delta)\oplus B_{H^+_{\lambda^\ast}}(0,\epsilon_1/2)\oplus
B_{H^-_{\lambda^\ast}}(0,\epsilon_1/2)$. Finally,  take $\epsilon=\min\{\epsilon_1/2, \delta\}$ and let
$\mathcal{W}$ be the image of
$\Lambda_0\times {B}_{H^0_{\lambda^\ast}}(0,\epsilon)\oplus B_{H^+_{\lambda^\ast}}(0,\epsilon)\oplus
B_{H^-_{\lambda^\ast}}(0,\epsilon)$ under $\Phi$.
Then for $\lambda\in\Lambda_0$, $\mathcal{W}_\lambda:=\{v\in H\,|\, (\lambda,v)\in\mathcal{W}\}$
is an open neighborhood of $0$ in $H$, and $\Phi_\lambda$ is
a homeomorphism from ${B}_{H^0_{\lambda^\ast}}(0,\epsilon)\oplus B_{H^+_{\lambda^\ast}}(0,\epsilon)\oplus
B_{H^-_{\lambda^\ast}}(0,\epsilon)$ onto ${\mathcal{W}}_{\lambda}$.
\end{proof}

As an application of Theorem~\ref{th:A.7} we have the following sufficient criterion for bifurcations,
which improves \cite[Theorem~3.6]{Lu8}. Of course, there also exist corresponding versions of \cite[Corollaries~3.7,~3.8]{Lu8}.

\begin{theorem}\label{th:A.9}
In the assumptions of Theorem~\ref{th:A.7},
 if $\Lambda$ is an open internal in $\mathbb{R}$
and  there exist two points in any neighborhood of $\lambda^\ast\in\Lambda$, $\lambda_-<\lambda^\ast<\lambda_+$,
 such that
\begin{equation}\label{e:newbifur}
 \mu_{\lambda_-}\ne \mu_{\lambda_+}\quad\hbox{and}\quad
\nu_{\lambda_-}=\nu_{\lambda_+}=0.
\end{equation}
(Here $\mu_{\lambda}=\dim H^-_\lambda$ and $\nu_{\lambda}=\dim H^0_\lambda$ are dimensions of
the negative definite and zero spaces of $B_\lambda(0)$, respectively.)
  Then  $(\lambda^\ast, 0)$ is a bifurcation point  of $A_\lambda(x)=0$  in $\Lambda\times X$.
\end{theorem}
\begin{proof}[\bf Proof]
Firstly, note that  the condition ``${\rm Ker}(B_{\lambda^\ast}(0))\ne\{0\}$'' in the present case is implied in the second assumption
in (\ref{e:newbifur}).
(Otherwise, by Remark~\ref{rm:A.8}(iv) we deduce that all $\mathcal{L}_\lambda$ for $\lambda$ near $\lambda^\ast$ have the same Morse indexes at $0$,
and therefore $\lambda\mapsto \mu_{\lambda}$ is constant near $\lambda^\ast$.
This contradicts the first assumption in (\ref{e:newbifur}).)

Therefore we have (\ref{e:Spli.2.1.1})-(\ref{e:Spli.2.5}).
 By  (\ref{e:Spli.2.1.2}) and (\ref{e:Spli.2.4}),   for each $\lambda\in\Lambda_0$,
the map $z\mapsto z+ \psi({\lambda}, z))$ induces an one-to-one correspondence
 between the critical points of  $\mathcal{L}_{\lambda}^\circ$ near $0\in H^0_{\lambda^\ast}={\rm Ker}(B_{\lambda^\ast}(0))$
and zeros of $A_{\lambda}$ near $0\in X$.

\textsf{By a contradiction, suppose that $(\lambda^\ast, 0)\in\Lambda\times X$
 is not a bifurcation point of $A_\lambda(x)=0$ in $\Lambda\times X$.}
 Then by shrinking $\Lambda_0$ toward $\lambda^\ast$ and $\epsilon>0$,  for each
$\lambda\in \Lambda_0$ the functional $\mathcal{L}^\circ_\lambda$ has a unique
critical point $0$ sitting in ${B}_{H^0_{\lambda^\ast}}(0, \epsilon)\subset X$.
Note that (\ref{e:Spli.2.3}) and (\ref{e:Spli.2.4}) imply
\begin{eqnarray*}
&&\Lambda_0\times \bar{B}_{H^0_{\lambda^\ast}}(0, \epsilon/2)\ni (\lambda, z)\mapsto\mathcal{L}_{\lambda}^\circ(z)\in\R\quad\hbox{and}\\
&&\Lambda_0\times \bar{B}_{H^0_{\lambda^\ast}}(0, \epsilon/2)\ni (\lambda, z)\mapsto d\mathcal{L}^\circ_\lambda(z)\in H_{\lambda^\ast}^0=X_{\lambda^\ast}^0
\end{eqnarray*}
 are uniformly continuous (because we can assume $\Lambda_0$ to be compact). By \cite[Theorem 5.1]{CorH} (see \cite[Theorem~2.2]{Lu8})
 we obtain that critical groups
 \begin{eqnarray}\label{e:criticalgroup}
 \hbox{$C_\ast(\mathcal{L}^\circ_\lambda, 0;{\bf K})$ are independent of $\lambda\in \Lambda_0$ for any Abel group ${\bf K}$.}
\end{eqnarray}
Since  \cite[Corollary~A.6]{Lu8} gives rise to
\begin{eqnarray*}
C_q(\mathcal{L}_{{\lambda}}, 0;{\bf K})=C_{q-\mu_{\lambda^\ast}}(\mathcal{L}^\circ_{{\lambda}}, 0;{\bf K}),\quad\forall
q\in\mathbb{N}\cup\{0\},
\end{eqnarray*}
it follows from this and (\ref{e:criticalgroup}) that
 \begin{eqnarray}\label{e:criticalgroup*}
 \hbox{$C_\ast(\mathcal{L}_\lambda, 0;{\bf K})$ are independent of $\lambda\in \Lambda_0$ for any Abel group ${\bf K}$.}
\end{eqnarray}
On the other hand, by the assumptions  there exist sequences
$(\lambda_k^-)\subset (-\infty, \lambda^\ast)\cap\Lambda_0$ and $(\lambda_k^+)\subset (\lambda^\ast, +\infty)\cap\Lambda_0$
converging to $\lambda^\ast$ such that
$\nu_{\lambda_k^-}=0=\nu_{\lambda_k^+}$ and $\mu_{\lambda_k^-}\ne \mu_{\lambda_k^+}$
for all $k\in\N$.
From the first two equalities and \cite[(2.7)]{Lu3} we derive that
$$
C_q(\mathcal{L}_{{\lambda_k^+}}, 0;{\bf K})=\delta^q_{\mu_{\lambda_k^+}}{\bf K}\quad\hbox{and}\quad
C_q(\mathcal{L}_{{\lambda_k^-}}, 0;{\bf K})=\delta^q_{\mu_{\lambda_k^-}}{\bf K},\quad\forall k\in\N.
$$
Hereafter $\delta^q_p=1$ if $p=q$, and $\delta^q_p=0$ if $p\ne q$.
But these and (\ref{e:criticalgroup*}) imply that $\mu_{\lambda_k^+}= \mu_{\lambda_k^-}$ for each $k\in\N$.
A contradiction is yielded.
\end{proof}

\noindent{\bf 3.2}. {\bf A few of bifurcation theorems of Rabinowitz  or Fadell-Rabinowitz type.}
In this subsection, except that Theorem~\ref{th:A.11} is of Fadell-Rabinowitz type, others are of Rabinowitz type.
By Theorems~\ref{th:A.1.1},~\ref{th:A.7} we get the following improvement of \cite[Theorem~4.6]{Lu8} immediately.

\begin{theorem}\label{th:A.10}
Let $H$, $X$ and $U$ be as in Hypothesis~\ref{hyp:A.5},
and let $\{\mathcal{L}_\lambda\in C^1(U, \mathbb{R})\,|\,\lambda\in\Lambda\}$ be a continuous family of functionals
parameterized by an open interval $\Lambda\subset\mathbb{R}$ containing  $\lambda^\ast$.
 For each $\lambda\in\Lambda$, assume $\mathcal{L}'_\lambda(0)=0$, and that
  there exists a map $A_\lambda\in C^1(U^X, X)$ such that $\Lambda\times U^X\ni (\lambda, x)\to A_\lambda(x)\in X$ is continuous,
and that
 $$
 D\mathcal{L}_\lambda(x)[u]=(A_\lambda(x), u)_H\quad\hbox{and}\quad
(DA_\lambda(x)[u], v)_H=(B_\lambda(x)u, v)_H
$$
 for all $x\in U\cap X$  and $u, v\in X$.
 Suppose also that the following conditions hold.
 \begin{description}
 \item[(a)]   $B_\lambda$ has a decomposition
$B_\lambda=P_\lambda+Q_\lambda$, where for each $x\in U\cap X$,
 $P_\lambda(x)\in\mathscr{L}_s(H)$ is  positive definitive and
$Q_\lambda(x)\in\mathscr{L}_s(H)$ is compact, so that
$(\mathcal{L}_{\lambda}, H, X, U, A_{\lambda}, B_{\lambda}=P_{\lambda}+ Q_{\lambda})$ satisfies Hypothesis~\ref{hyp:A.5}.

\item[(b)]  For each $h\in H$, it holds that $\|P_{\lambda}(x)h-P_{\lambda^\ast}(0)h\|\to 0$
as $x\in U\cap X$ approaches to $0$ in $H$ and $\lambda\in\Lambda$ converges to $\lambda^\ast$.

 \item[(c)]  For some small $\delta>0$, there exists a positive constant $c_0>0$ such that
$$
(P_\lambda(x)u, u)\ge c_0\|u\|^2,\quad\forall u\in H,\;\forall x\in
\bar{B}_H(0,\delta)\cap X,\quad\forall\lambda\in \Lambda.
$$
 \item[(d)]  $Q_\lambda: U\cap X\to \mathscr{L}_s(H)$ is uniformly continuous at $0$  with respect to $\lambda\in \Lambda$.
  \item[(e)]  If $\lambda\in \Lambda$ converges to $\lambda^\ast$ then
  $\|Q_{\lambda}(0)-Q_{\lambda^\ast}(0)\|\to 0$.
   \item[(f)]  $\nu_{\lambda^\ast}>0$, $\nu_{\lambda}=0$ for any $\lambda\in\Lambda\setminus\{\lambda^\ast\}$,
   and the Morse indexes of $\mathcal{L}_\lambda$
at $0\in H$ take values $\mu_{\lambda^\ast}$ and $\mu_{\lambda^\ast}+\nu_{\lambda^\ast}$
 as $\lambda\in\mathbb{R}$ varies in both sides of $\lambda^\ast$ and is close to $\lambda^\ast$,
where $\mu_{\lambda}$ and $\nu_{\lambda}$ are the Morse index and the nullity of  $\mathcal{L}_{\lambda}$
at $0$, respectively.
    \end{description}
 Then  $(\lambda^\ast,0)\in\Lambda\times U^X$ is a bifurcation point  for the equation
 \begin{equation}\label{e:Bif.2.2.1*}
A_\lambda(u)=0,\quad (\lambda, u)\in\Lambda\times U^X;
\end{equation}
  in particular, $(\lambda^\ast,0)\in\Lambda\times U$ is a bifurcation point  for the equation
\begin{equation*}
D\mathcal{L}_\lambda(u)=0,
\quad (\lambda,u)\in \Lambda\times U.
\end{equation*}
More precisely, one of the following alternatives occurs:
\begin{description}
\item[(i)] $(\lambda^\ast,0)$ is not an isolated solution  in  $\{\lambda^\ast\}\times U^X$ of the equation (\ref{e:Bif.2.2.1*}).

\item[(ii)]  For every $\lambda\in\Lambda$ near $\lambda^\ast$ there is a nontrivial solution $u_\lambda$ of (\ref{e:Bif.2.2.1*}) in $U^X$,
which  converges to $0$ in $X$ as $\lambda\to\lambda^\ast$.

\item[(iii)] For any given neighborhood $W$ of $0$ in $X$ there is an one-sided  neighborhood $\Lambda^\ast$ of $\lambda^\ast$ such that
for any $\lambda\in\Lambda^\ast\setminus\{\lambda^\ast\}$, (\ref{e:Bif.2.2.1*}) has  at least two  nontrivial solutions in $W$,
which can also be required to correspond to distinct critical values
provided that  $\nu_{\lambda^\ast}>1$ and (\ref{e:Bif.2.2.1*}) has only finitely many nontrivial solutions in $W$.
\end{description}
 \end{theorem}

\begin{proof}[\bf Proof]
The first conclusion follows from Theorem~\ref{th:A.9}.
Applying Theorem~\ref{th:A.1.1} to the functionals $\mathcal{L}_{\lambda}^\circ$ in (\ref{e:Spli.2.3})
may yield the claims after ``More precisely''.
\end{proof}

From Theorem~\ref{th:A.7} and \cite[\S4]{BaCl} (cf. \cite[Theorem~5.11]{Lu8}) we may also obtained
the following improvement of \cite[Theorem~5.12]{Lu8}.

 \begin{theorem}\label{th:A.11}
Under the assumptions of Theorem~\ref{th:A.10} let $H$ be equipped with
a continuous action of  a compact Lie group $G$ via Hilbert space isomorphisms on $H$.
Suppose that the action of $G$ on $H$ induces
 a continuous action on $X$ via Banach space isomorphisms on $X$,
 and that both $U$ and $\mathcal{L}_\lambda$ are $G$-invariant (and hence $H^0_\lambda$, $H^+_\lambda$ and $H^-_\lambda$
are $G$-invariant subspaces).
If the fixed point set of the induced $G$-action on $H^0_{\lambda^\ast}$ is $\{0\}$ then
 one of the following alternatives occurs:
\begin{description}
\item[(i)] $(\lambda^\ast,0)$ is not an isolated solution  in  $\{\lambda^\ast\}\times U^X$ of the equation (\ref{e:Bif.2.2.1*});
\item[(ii)] there exist left and right  neighborhoods $\Lambda^-$ and $\Lambda^+$ of $\lambda^\ast$ in $\mathbb{R}$
and integers $n^+, n^-\ge 0$, such that $n^++n^-\ge \ell(SH^0_{\lambda^\ast})$
and for $\lambda\in\Lambda^-\setminus\{\lambda^\ast\}$ (resp. $\lambda\in\Lambda^+\setminus\{\lambda^\ast\}$),
$\mathcal{L}_\lambda$ has at least $n^-$ (resp. $n^+$) distinct critical
$G$-orbits different from $0$, which converge to
 $0$ as $\lambda\to\lambda^\ast$.
\end{description}
In particular,  $(\lambda^\ast, 0)\in [\lambda^\ast-\delta, \lambda^\ast+\delta]\times U^X$
is a bifurcation point of (\ref{e:Bif.2.2.1*}).
\end{theorem}

\begin{remark}\label{rem:Bif.3.4.0}
{\rm  Because of the second claim in the bracket of the conclusion (C) of Theorem~\ref{th:A.7}  this theorem weakens the original assumption ``
 which induces  $C^1$ isometric actions on $X$'' in \cite[Theorem~5.12]{Lu8}.
Moreover, by \cite[Theorem~1]{CherMar70}, the continuity of $G\times X\ni (g,x)\mapsto g\cdot x\in X$
may be replaced by the condition that $G\ni g\mapsto g\cdot x\in X$ is continuous for any $x\in X$.
}
\end{remark}

 \begin{remark}\label{rem:Bif.3.4}
{\rm By \cite[Propositions~2.4,2.6]{BaCl} (cf. \cite[Remark~5.14]{Lu8}), in Theorem~\ref{th:A.11}
\begin{enumerate}
\item[(i)] if $G=(\mathbb{Z}/p\mathbb{Z})^r$, where $r>0$ and $p$ is a prime,
 taking the Borel cohomology $H^\ast_G$ with coefficients in $G=(\mathbb{Z}/p\mathbb{Z})^r$
one gets that $\ell(SH^0_{\lambda^\ast})$  is equal
to $\dim H^0_{\lambda^\ast}$ for $p=2$, and to $\frac{1}{2}\dim H^0_{\lambda^\ast}$ for $p>2$;
\item[(ii)] if $G=(S^1)^r$,  $r>0$,  taking the $\mathbb{Q}$-coefficients Borel cohomology $H^\ast_G$ we get
 $\ell(SH^0_{\lambda^\ast})=\frac{1}{2}\dim H^0_{\lambda^\ast}$;
\item[(iii)] if $G=S^1\times\Gamma$, $\Gamma$ is finite, and such that the fixed point set of $S^1\equiv S^1\times\{e\}$
is trivial,  taking the $\mathbb{Q}$-coefficients Borel cohomology $H^\ast_G$ we have $\ell(SH^0_{\lambda^\ast})=\frac{1}{2}\dim H^0_{\lambda^\ast}$.
\end{enumerate}
}
\end{remark}

By Remark~\ref{rem:Bif.3.4} it is not hard to see that the result derived from Theorems~\ref{th:A.2},~\ref{th:A.7}
is contained in Theorem~\ref{th:A.11}. However,  from Theorems~\ref{th:A.3},~\ref{th:A.7}
we may derive the following theorem, which generalizes
Theorem~\ref{th:A.10} (\cite[Theorem~4.6]{Lu8}). But it and Theorem~\ref{th:A.11} cannot be contained each other.

\begin{theorem}\label{th:A.13}
In Theorem~\ref{th:A.11}, if the assumption ``the fixed point set of the induced $G$-action on $H^0_{\lambda^\ast}$ is $\{0\}$''
is removed, then $(\lambda^\ast,0)\in\Lambda\times U^X$ is a bifurcation point  for the equation
(\ref{e:Bif.2.2.1*}); in particular, $(\lambda^\ast,0)\in\Lambda\times U$ is a bifurcation point  for the equation
\begin{equation*}
D\mathcal{L}_\lambda(u)=0,
\quad (\lambda,u)\in \Lambda\times U.
\end{equation*}
More precisely, one of the following alternatives occurs:
\begin{description}
\item[(i)] $(\lambda^\ast,0)$ is not an isolated solution  in  $\{\lambda^\ast\}\times U^X$ of the equation
(\ref{e:Bif.2.2.1*}).

\item[(ii)]  For every $\lambda\in\Lambda\setminus\{\lambda^\ast\}$ near $\lambda^\ast$ there is a nontrivial $G$-orbit of solutions of (\ref{e:Bif.2.2.1*}) in $U^X$,
which  converges to $0$ in $X$ as $\lambda\to\lambda^\ast$.

\item[(iii)] For any given $G$-invariant neighborhood $\mathcal{N}$ of $0$ in $X$
there is an one-sided neighborhood $\Lambda^0$ of $\lambda^\ast$ such that
for any $\lambda\in\Lambda^0\setminus\{\lambda^\ast\}$,
(\ref{e:Bif.2.2.1*})  has at least two nontrivial $G$-orbit of solutions in $\mathcal{N}$
 provided that the Euler-Poincar\'e characteristic of any nontrivial orbit near $0$ of the induced $G$-action on
 $H^0_{\lambda^\ast}$ is not equal to $1-(-1)^{\nu_{\lambda^\ast}}$, where $\nu_{\lambda^\ast}=\dim H^0_{\lambda^\ast}$
 is the nullity of $\mathcal{L}_{\lambda^\ast}$ at $0$. Moreover,
 for $\lambda\in\Lambda^0\setminus\{\lambda^\ast\}$, if (\ref{e:Bif.2.2.1*})
 has only finitely many  $G$-orbit of solutions in $\mathcal{N}$, then
it has at least two nontrivial $G$-orbit of solutions in $\mathcal{N}$ with different energy
 provided that  $\nu_{\lambda^\ast}>1$ and
 any nontrivial orbit $\mathcal{O}$ near $0$ of the induced $G$-action on $H^0_{\lambda^\ast}$ satisfies one of the following conditions:
  \begin{description}
\item[iii-1)]  $\dim\mathcal{O}=0$ or  $1\le\dim\mathcal{O}\le \nu_{\lambda^\ast}-2$.
\item[iii-2)]  $1\le \dim \mathcal{O}_1=\nu_{\lambda^\ast}-1$, either $\mathcal{O}$ is non-connected or
$\mathcal{O}$ is connected and $H_r(\mathcal{O}, \mathbb{Z}_2)\ne H_r(S^{\nu_{\lambda^\ast}-1}, \mathbb{Z}_2)$ for some $0\le r\le \nu_{\lambda^\ast}-1$.
\end{description}
\end{description}
\end{theorem}
\begin{proof}[\bf Proof]
The first conclusion follows from Theorem~\ref{th:A.9}.
Let us prove others.
By the conclusion (C) in Theorem~\ref{th:A.7}
we have an induced $C^\infty$
 $G$-action on $H^0_{\lambda^\ast}$ via Hilbert space isomorphisms,
 and for each $\lambda\in \Lambda$,
the  maps $\psi(\lambda, \cdot)$  and $\Phi_{\lambda}(\cdot,\cdot)$  in (\ref{e:Spli.2.1.1}) and (\ref{e:Spli.2.1.3}) are
  $G$-equivariant, and $\mathcal{L}^\circ_{\lambda}$ in (\ref{e:Spli.2.3}) is $G$-invariant.
 Clearly, $0\in H^0_{\lambda^\ast}$ belongs to the set of fixed points of the induced $C^\infty$ $G$-action.
As in the proof of \cite[Theorem~5.12]{Lu8} we obtain either
\begin{eqnarray}\label{e:Bif.2.2.3}
0\in H^0_{\lambda^\ast}\;\hbox{is a strict local}\left\{
\begin{array}{ll}
\hbox{minimizer of}\;\mathcal{L}^\circ_{\lambda},&\quad
\forall \lambda\in [\lambda^\ast-\delta, \lambda^\ast),\\
\hbox{maximizer of}\;\mathcal{L}^\circ_{\lambda},&\quad
\forall \lambda\in (\lambda^\ast, \lambda^\ast+\delta]
\end{array}\right.
\end{eqnarray}
or
\begin{eqnarray}\label{e:Bif.2.2.4}
0\in H^0_{\lambda^\ast}\;\hbox{is a strict local}\left\{
\begin{array}{ll}
\hbox{maximizer of}\;\mathcal{L}^\circ_{\lambda},&\quad
\forall \lambda\in [\lambda^\ast-\delta, \lambda^\ast),\\
\hbox{minimizer of}\;\mathcal{L}^\circ_{\lambda},&\quad
\forall \lambda\in (\lambda^\ast, \lambda^\ast+\delta].
\end{array}\right.
\end{eqnarray}
Suppose that any of the conclusions (i)-(ii) does not hold. Then $0\in H$ is
an isolated critical point of $\mathcal{L}_{\lambda^\ast}$ and so
$0\in H^0_{\lambda^\ast}$ is also an isolated critical point of $\mathcal{L}^\circ_{\lambda^\ast}$.
By the assumptions, the Euler-Poincar\'e characteristic of any nontrivial orbit near $0$
 of the induced $G$-action on $H^0_{\lambda^\ast}$ is not equal to $1-(-1)^{\nu_{\lambda^\ast}}$,
and  any nontrivial orbit $\mathcal{O}$ near $0$
  of the induced $G$-action on $H^0_{\lambda^\ast}$ satisfies
 one of the above conditions iii-1) and iii-2) if $\nu_{\lambda^\ast}>1$.
 Applying Theorem~\ref{th:A.3} to the family of functionals in (\ref{e:Spli.2.3})
we obtain:

\textsf{There exists a small $G$-invariant neighborhood $W$ of $0$ in $B_{H^0_{\lambda^\ast}}(0,\epsilon)$
  and an one-sided  neighborhood
 $\Lambda^0$ of $\lambda^\ast$ such that $\psi(\lambda,W)\subset\mathcal{N}$ for all $\lambda\in\Lambda^0$
 and that for every $\lambda\in\Lambda^0\setminus\{\lambda^\ast\}$ there holds:\\
 (a) the functional $\mathcal{L}_{\lambda}^\circ$ has at least two  nontrivial critical orbits  in $W$,
 $\mathcal{O}_i^\ast$, $i=1,2$.\\ 
 (b) If $\nu_{\lambda^\ast}>1$ and $\mathcal{L}_{\lambda}^\circ$ has only finitely many critical orbits in $W$,
 then the orbits $\mathcal{O}_i^\ast$, $i=1,2$,  can be chosen to satisfy
 $\mathcal{L}_{\lambda}^\circ|_{\mathcal{O}_2^\ast}\ne\mathcal{L}_{\lambda}^\circ|_{\mathcal{O}_1^\ast}$.}

For a critical orbit $\mathcal{O}^\ast$ of $\mathcal{L}_{\lambda}^\circ$ in $W$ and any $z\in\mathcal{O}^\ast$,
by Theorem~\ref{th:A.7}, $\mathcal{O}:=G(\psi(\lambda, z))$ is a critical orbit of $\mathcal{L}_{\lambda}$ sitting in
$\mathcal{N}$ and
$\mathcal{L}_{\lambda}|_{\mathcal{O}}=\mathcal{L}_{\lambda}^\circ|_{\mathcal{O}^\ast}$.
Therefore $\mathcal{O}_i:=G(\psi(\lambda, z_i))$ with $z_i\in \mathcal{O}^\ast_i$, $i=1,2$, satisfy the claims in iii).

\end{proof}

{\bf Note}:  In Theorem~\ref{th:A.13}, if $G$ is a finite group, by Theorem~\ref{th:A.10}
we see that the sentence ``(\ref{e:Bif.2.2.1*})  has at least two nontrivial $G$-orbit of solutions in $\mathcal{N}$
 provided that the Euler-Poincar\'e characteristic of any nontrivial orbit near $0$ of the induced $G$-action on
 $H^0_{\lambda^\ast}$ is not equal to $1-(-1)^{\nu_{\lambda^\ast}}$, where $\nu_{\lambda^\ast}=\dim H^0_{\lambda^\ast}$
 is the nullity of $\mathcal{L}_{\lambda^\ast}$ at $0$'' in Theorem~\ref{th:A.13}(iii)
may be replaced by ``(\ref{e:Bif.2.2.1*})  has at least two nontrivial solutions in $\mathcal{N}$''.
Similar replacements also hold for the following theorems and corollaries.\\

There is also a corresponding corollary to \cite[Corollary~5.13]{Lu8}.
Moreover, if ``Theorem~\ref{th:A.10}'' (which is implied in the assumptions of Theorem~\ref{th:A.11})
in Theorem~\ref{th:A.13} is replaced by ``\cite[Theorem~6.1]{Lu8}''
the  conclusions are still true.

However, if the reduced functionals on a Banach space of finite dimension are only $C^1$,
the result in \cite[\S4]{BaCl} (cf. \cite[Theorem~5.11]{Lu8}) cannot be used. It is possible
for us to use Theorems~\ref{th:A.3},~\ref{th:A.2}.

\begin{hypothesis}[{\cite[Hypothesis~1.2]{Lu8}}]\label{hyp:A.14}
{\rm Let $U\subset H$ be as in Hypothesis~\ref{hyp:A.4},  $\mathcal{L}\in C^1(U,\mathbb{R})$ satisfy
$\mathcal{L}'(0)=0$  and the gradient $\nabla\mathcal{L}$ have the G\^ateaux derivative
$\mathcal{L}''(u)\in\mathscr{L}_s(H)$ at any $u\in U$, which is a compact operator
 and approaches to $\mathcal{L}''(0)$ in $\mathscr{L}_s(H)$ as $u\to 0$ in $H$.}
\end{hypothesis}

\begin{theorem}\label{th:A.15}
Let $\mathcal{L}\in C^1(U,\mathbb{R})$ (resp.  $\widehat{\mathcal{L}}\in C^1(U,\mathbb{R})$) satisfy
 Hypothesis~\ref{hyp:A.4} with $X=H$ (resp.  Hypothesis~\ref{hyp:A.14}),
and let $\lambda^\ast\in\mathbb{R}$ be an isolated eigenvalue of
\begin{eqnarray*}
\mathcal{L}''(0)v-\lambda\widehat{\mathcal{L}}''(0)v=0,\quad v\in H.
\end{eqnarray*}
(If $\lambda^\ast=0$, it is enough that $\widehat{\mathcal{L}}\in C^1(U,\mathbb{R})$ satisfies Hypothesis~\ref{hyp:A.14}
without requirement that each $\widehat{\mathcal{L}}''(u)\in\mathscr{L}_s(H)$ is compact.)
Assume that $H$ is equipped with a continuous action of  a compact Lie group $G$ via Hilbert space isomorphism on $H$
such that  $U$ and $\mathcal{L}, \widehat{\mathcal{L}}$ are $G$-invariant.
Suppose that  the Morse indexes of $\mathcal{L}_\lambda:=\mathcal{L}-\lambda\widehat{\mathcal{L}}$
at $0\in H$ take values $\mu_{\lambda^\ast}$ and $\mu_{\lambda^\ast}+\nu_{\lambda^\ast}$
 as $\lambda\in\mathbb{R}$ varies in both sides of $\lambda^\ast$ and is close to $\lambda^\ast$,
where $\mu_{\lambda}$ and $\nu_{\lambda}$ are the Morse index and the nullity of  $\mathcal{L}_{\lambda}$
at $0$, respectively.  Then  $(\lambda^\ast, 0)\in\mathbb{R}\times U$ is a bifurcation point  for the equation
(\ref{e:Bi.2.7.3}), and one of the following alternatives occurs:
\begin{description}
\item[(i)] $(\lambda^\ast, 0)$ is not an isolated solution in $\{\lambda^\ast\}\times U$ of
\begin{eqnarray}\label{e:Bi.2.7.3}
\mathcal{L}'(u)=\lambda\widehat{\mathcal{L}'}(u).
\end{eqnarray}
 \item[(ii)]  For every $\lambda\in\mathbb{R}$ near $\lambda^\ast$ there is a nontrivial $G$-orbit of solutions
 of (\ref{e:Bi.2.7.3}) in $U$, which converges to $0$ as $\lambda\to\lambda^\ast$;

\item[(iii)] For any given $G$-invariant neighborhood $\mathcal{N}$ of $0$ in $U$
there is an one-sided neighborhood $\Lambda^0$ of $\lambda^\ast$ in $\mathbb{R}$ such that
for any $\lambda\in\Lambda^0\setminus\{\lambda^\ast\}$,  (\ref{e:Bi.2.7.3})
 has at least two nontrivial $G$-orbit of solutions in $\mathcal{N}$
 provided that the Euler-Poincar\'e characteristic of any nontrivial orbit near $0$ of the induced $G$-action on
 $H^0_{\lambda^\ast}:={\rm Ker}(\mathcal{L}''(0)-\lambda^\ast\widehat{\mathcal{L}}''(0))$
   (which is, by a result in \cite{BoMon45}, a $C^\infty$ $G$-action because $\dim H^0_{\lambda^\ast}<\infty$)
  is not equal to $1-(-1)^{\nu_{\lambda^\ast}}$, where $\nu_{\lambda^\ast}=\dim H^0_{\lambda^\ast}$
 is the nullity of $\mathcal{L}_{\lambda^\ast}$ at $0$.
 Moreover,  for $\lambda\in\Lambda^0\setminus\{\lambda^\ast\}$, if (\ref{e:Bi.2.7.3})
 has only finitely many  $G$-orbit of solutions in $\mathcal{N}$, then
it has at least two nontrivial $G$-orbit of solutions in $\mathcal{N}$ with different energy
 provided that  $\nu_{\lambda^\ast}>1$ and
 any nontrivial orbit $\mathcal{O}$ near $0$ of the induced $G$-action on $H^0_{\lambda^\ast}$ satisfies one of the following conditions:
  \begin{description}
\item[iii-1)]  $\dim\mathcal{O}=0$ or  $1\le\dim\mathcal{O}\le \nu_{\lambda^\ast}-2$.
\item[iii-2)]  $1\le \dim \mathcal{O}=\nu_{\lambda^\ast}-1$, either $\mathcal{O}$ is non-connected or
$\mathcal{O}$ is connected and $H_r(\mathcal{O}, \mathbb{Z}_2)\ne H_r(S^{\nu_{\lambda^\ast}-1}, \mathbb{Z}_2)$ for some $0\le r\le \nu_{\lambda^\ast}-1$.
\end{description}
\end{description}
\end{theorem}
\begin{proof}[\bf Proof]
\cite[Theorem~4.2]{Lu8} gives the first claim.
In order to prove others, let $\mathcal{L}^\circ_\lambda$ be as in \cite[(4.5)]{Lu8}, (which was obtained by \cite[Theorem~2.16]{Lu7}),   i.e.,
\begin{equation}\label{e:Bi.2.15}
 \mathcal{L}^\circ_\lambda: B_H(0, \epsilon)\cap H^0_{\lambda^\ast}\to\mathbb{R},\;z\mapsto
 \mathcal{L}(z+\psi(\lambda, z))- \lambda\widehat{\mathcal{L}}(z+\psi(\lambda, z)),
  \end{equation}
where $\psi:[\lambda^\ast-\delta, \lambda^\ast+\delta]\times (B_H(0,\epsilon)\cap H^0_{\lambda^\ast})\to (H^0_{\lambda^\ast})^\bot$
is a unique continuous map satisfying
\begin{eqnarray*}
 P^\bot_{\lambda^\ast}\nabla\mathcal{L}(z+ \psi(\lambda, z))-
 \lambda P^\bot_{\lambda^\ast}\nabla\widehat{\mathcal{L}}(z+ \psi(\lambda, z))=0\quad\forall z\in B_{H}(0,\epsilon)\cap H^0_{\lambda^\ast}.
 \end{eqnarray*}
A point $z\in B_{H}(0,\epsilon)\cap H^0_{\lambda^\ast}$ is a critical point of  $\mathcal{L}^\circ_\lambda$
if and only if $z+\psi(\lambda,z)$ is a critical point of $\mathcal{L}_\lambda=\mathcal{L}- \lambda\widehat{\mathcal{L}}$
near $0\in H$. It was proved in \cite[(4.9), (4.11)]{Lu8} that
(\ref{e:Bif.2.2.3}) and (\ref{e:Bif.2.2.4}) hold for these $\mathcal{L}^\circ_\lambda$.
Note that $0\in H^0_{\lambda^\ast}$ is a fixed point for the induced $G$-action on $H^0_{\lambda^\ast}$.
As in the proof of Theorem~\ref{th:A.13} the conclusions may follow from Theorem~\ref{th:A.3}.
\end{proof}

Because $\mathcal{L}^\circ_\lambda$ in (\ref{e:Bi.2.15}) is only $C^1$,
 the result in \cite[\S4]{BaCl} (cf. \cite[Theorem~5.11]{Lu8}) cannot be applied to it.
\cite[Theorem~5.9]{Lu8} was obtained by applying \cite[Theorem~5.1]{Lu8} to it.
Therefore using Theorem~\ref{th:A.2} instead of \cite[Theorem~5.1]{Lu8} we may weaken the assumption
``a linear isometric action of a compact  Lie group $G$'' in \cite[Theorem~5.1]{Lu8}
as ``a continuous action $\pi$ of a compact Lie group $G$ via linear isometries''.

Corresponding to \cite[Corollary~4.3]{Lu8} and \cite[Corollary~4.4]{Lu8}, we have:

\begin{corollary}\label{cor:Bi.2.4.1}
Let $\mathcal{L}\in C^1(U,\mathbb{R})$ (resp.  $\widehat{\mathcal{L}}\in C^1(U,\mathbb{R})$) satisfy
 Hypothesis~\ref{hyp:A.4} with $X=H$ (resp.  Hypothesis~\ref{hyp:A.14}),
and let $\lambda^\ast\in\mathbb{R}$ be an isolated eigenvalue of
\begin{eqnarray}\label{e:Bi.2.7.4}
\mathcal{L}''(0)v-\lambda\widehat{\mathcal{L}}''(0)v=0,\quad v\in H.
\end{eqnarray}
(If $\lambda^\ast=0$, it is enough that $\widehat{\mathcal{L}}\in C^1(U,\mathbb{R})$ satisfies Hypothesis~\ref{hyp:A.14}
without requirement that each $\widehat{\mathcal{L}}''(u)\in\mathscr{L}_s(H)$ is compact.)
Suppose that $\widehat{\mathcal{L}}''(0)$ is either semi-positive or semi-negative.
Assume that $H$ is equipped with a continuous action of  a compact Lie group $G$ via Hilbert space isomorphism on $H$
 such that  $U$ and $\mathcal{L}, \widehat{\mathcal{L}}$ are $G$-invariant.
Then the conclusions of Theorem~\ref{th:A.15} hold true.
\end{corollary}

\begin{corollary}\label{cor:Bi.2.4.2}
Let $\mathcal{L}\in C^1(U,\mathbb{R})$ (resp.  $\widehat{\mathcal{L}}\in C^1(U,\mathbb{R})$) satisfy
 Hypothesis~\ref{hyp:A.4} with $X=H$ (resp.  Hypothesis~\ref{hyp:A.14}). Suppose that the following two conditions
  satisfied:
 \begin{enumerate}
\item[\rm (a)] $\mathcal{L}''(0)$ is invertible
and $\lambda^\ast=\lambda_{k_0}$ is an eigenvalue of (\ref{e:Bi.2.7.4}).
 \item[\rm (b)] $\mathcal{L}''(0)\widehat{\mathcal{L}}''(0)=\widehat{\mathcal{L}}''(0)\mathcal{L}''(0)$ (so
 each $H_k$  is an invariant subspace of $\mathcal{L}''(0)$),   and $\mathcal{L}''(0)$
 is either positive  or negative on $H_{k_0}$.
 \end{enumerate}
 Assume that $H$ is equipped with a continuous action of  a compact Lie group $G$ via Hilbert space isomorphism on $H$
  such that  $U$ and $\mathcal{L}, \widehat{\mathcal{L}}$ are $G$-invariant.
Then the conclusions of Theorem~\ref{th:A.15} hold true.
 Moreover, if ${\mathcal{L}}''(0)$ is positive definite, the condition (b) is unnecessary.
\end{corollary}

\noindent{\bf 3.3}.
{\bf Improvements of \cite[Theorems~5.18,~5.19]{Lu8}.}

\begin{hypothesis}[\hbox{\cite[Hypothesis~2.20]{Lu7}}]\label{hyp:S.6.2}
{\rm {\bf (i)} Let $G$ be a compact
Lie group, and  $\mathcal{H}$  a $C^3$ Hilbert-Riemannian $G$-space
(that is, ${\mathcal{H}}$ is a $C^3$ Hilbert-Riemannian manifold equipped with
a $C^3$ action via  Riemannian isometries, see \cite{Was}).\\
 {\bf (ii)} The $C^1$ functional $\mathcal{ L}:\mathcal{H}\to\mathbb{R}$ is $G$-invariant, the gradient
 $\nabla\mathcal{L}:\mathcal{H}\to T\mathcal{H}$ is G\^ateaux differentiable
 (i.e., under any $C^3$ local chart the functional $\mathcal{L}$
 has a G\^ateaux differentiable gradient map), and $\mathcal{ O}$ is an isolated
 critical orbit which is a $C^3$ critical submanifold with  Morse index $\mu_\mathcal{O}$.}
\end{hypothesis}

 Under Hypothesis~\ref{hyp:S.6.2}
let $\pi:N\mathcal{ O}\to \mathcal{O}$ denote the normal bundle of $\mathcal{O}$. The
bundle is a $C^2$-Hilbert vector bundle over $\mathcal{O}$, and can
be considered as a subbundle of $T_\mathcal{O}{\mathcal{H}}$ via the
Riemannian metric $(\!(\cdot, \cdot)\!)$. The metric $(\!(\cdot, \cdot)\!)$
induces a natural $C^2$ orthogonal bundle
projection ${\bf \Pi}:T_{\mathcal{O}}\mathcal{H}\to N\mathcal{O}$. For $\varepsilon>0$,
the so-called normal disk bundle of radius $\varepsilon$ is denoted by
$N\mathcal{ O}(\varepsilon):=\{(x,v)\in N\mathcal{O}\,|\,\|v\|_{x}<\varepsilon\}$.
 If $\varepsilon>0$ is small enough  the exponential map $\exp$ gives a $C^2$-diffeomorphism
 $\digamma$ from  $N\mathcal{ O}(\varepsilon)$ onto an open
neighborhood of $\mathcal{ O}$ in ${\mathcal{H}}$, $\mathcal{
N}(\mathcal{ O},\varepsilon)$.
For $x\in\mathcal{ O}$, let  $\mathscr{L}_s(N\mathcal{O}_x)$ denote the space
of those operators $S\in \mathscr{L}(N\mathcal{ O}_x)$ which are self-adjoint
with respect to the inner product $(\!(\cdot, \cdot)\!)_x$, i.e.
$(\!(S_xu, v)\!)_x=(\!(u, S_xv)\!)_x$ for all $u, v\in N\mathcal{
O}_x$. Then we have a $C^2$ vector bundle $\mathscr{L}_s(N\mathcal{ O})\to
\mathcal{O}$ whose fiber at $x\in\mathcal{ O}$ is given by
$\mathscr{L}_s(N\mathcal{ O}_x)$.

\begin{hypothesis}[\hbox{\cite[Hypothesis~5.17]{Lu8}}]\label{hyp:Bi.3.19}
{\rm Under Hypothesis~\ref{hyp:S.6.2}, let for some $x_0\in\mathcal{O}$  the pair
$(\mathcal{L}\circ\exp|_{N\mathcal{O}(\varepsilon)_{x_0}},  N\mathcal{O}(\varepsilon)_{x_0})$
satisfy the corresponding conditions with Hypothesis~\ref{hyp:A.4} with $X=H=N\mathcal{O}(\varepsilon)_{x_0}$.
(For this goal we only need require that the pair $(\mathcal{L}\circ\exp_{x_0},  B_{T_{x_0}\mathcal{H}}(0,\varepsilon))$
satisfy the corresponding conditions with Hypothesis~\ref{hyp:A.4} with $X=H=T_{x_0}\mathcal{H}$
by \cite[Lemma~2.8]{Lu7}.) Let $\widehat{\mathcal{L}}\in C^1(\mathcal{H},\mathbb{R})$  be  $G$-invariant, have
a critical orbit $\mathcal{O}$, and also satisfy:
\begin{description}
\item[(i)] The gradient $\nabla(\widehat{\mathcal{L}}\circ\exp|_{B_{T_{x_0}\mathcal{H}}(0,\varepsilon)})$ is G\^ateaux differentiable, and its derivative
at any $u\in B_{T_{x_0}\mathcal{H}}(0,\varepsilon)$,
$$
d^2(\widehat{\mathcal{L}}\circ\exp|_{B_{T_{x_0}\mathcal{H}}(0,\varepsilon)})(u)\in\mathscr{L}_s(T_{x_0}\mathcal{H}),
$$
    is also a compact linear operator.
\item[(ii)]  $B_{T_{x_0}\mathcal{H}}(0,\varepsilon)\to \mathscr{L}_s(T_{x_0}\mathcal{H}),\;u\mapsto d^2(\widehat{\mathcal{L}}\circ\exp|_{B_{T_{x_0}\mathcal{H}}(0,\varepsilon)})(u)$
is continuous at $0\in T_{x_0}\mathcal{H}$.
(Thus the assumptions on $\mathcal{G}$ assure that the functionals
$\mathcal{L}_{\lambda}:=\mathcal{L}-\lambda\widehat{\mathcal{L}}$,
$\lambda\in\mathbb{R}$, also satisfy the conditions of
\cite[Theorems~2.21 and 2.22]{Lu7}.)
\end{description}}
\end{hypothesis}

Under Hypothesis~\ref{hyp:Bi.3.19}, we say $\mathcal{O}$ to be a  \textsf{bifurcation $G$-orbit
with parameter $\lambda^\ast$} of  the equation
\begin{equation}\label{e:Bi.3.11}
\mathcal{L}'(u)=\lambda\widehat{\mathcal{L}}'(u),\quad u\in \mathcal{H}
\end{equation}
if for any $\varepsilon>0$ and for any neighborhood $\mathscr{U}$ of $\mathcal{O}$ in $\mathcal{H}$
there exists a $G$-orbit of solutions  $\mathcal{O}'\ne \mathcal{O}$ in $\mathscr{U}$ of
(\ref{e:Bi.3.11}) with some $\lambda\in (\lambda^\ast-\varepsilon, \lambda^\ast+\varepsilon)$.
Equivalently, for some (and so any) fixed $x_0\in\mathcal{O}$ there exists a sequence $(\lambda_n, u_n)\subset (\lambda^\ast-\varepsilon, \lambda^\ast+\varepsilon)\times \mathcal{H}$ such that
\begin{equation}\label{e:Bi.3.11.1}
(\lambda_n, u_n)\to(\lambda^\ast, x_0),\quad \mathcal{L}'(u_n)=\lambda_n\widehat{\mathcal{L}}'(u_n)
\quad\hbox{and}\quad u_n\notin \mathcal{O}\quad\forall n.
\end{equation}

For any $x_0\in\mathcal{O}$, since  $\mathcal{S}_{x_0}:=\exp_{x_0}({N\mathcal{O}(\varepsilon)_{x_0}})$
is a $C^2$ slice for the action of $G$ on $\mathcal{H}$ (cf. \cite[page 1284]{Lu8})
a point $u\in {N\mathcal{O}(\varepsilon)_{x_0}}$ near $0_{x_0}\in {N\mathcal{O}(\varepsilon)_{x_0}}$
is a critical point of $\mathcal{L}_\lambda\circ\exp|_{N\mathcal{O}(\varepsilon)_{x_0}}$
if and only if $x:=\exp_{x_0}(u)$ is a critical point of  $\mathcal{L}_\lambda|_{\mathcal{S}_{x_0}}$.
Note that $d\mathcal{L}_\lambda(x)[\xi]=0\;\forall \xi\in T_{x}(G\cdot x)$ and
$T_x\mathcal{H}=T_{x}(G\cdot x)\oplus T_x\mathcal{S}_{x_0}$. We get that
$d\mathcal{L}_\lambda(x)=0$ with $x=\exp_{x_0}(u)$ if and only if
$d(\mathcal{L}_\lambda\circ\exp|_{N\mathcal{O}(\varepsilon)_{x_0}})(u)=0$.
Moreover, if $u_i\in N\mathcal{O}(\varepsilon)_{x_0}$, $i=1,2$, satisfies
$\exp_{x_0}(u_2)=g\exp_{x_0}(u_1)=\exp_{gx_0}(gu_1)$ for some $g\in G$,
since $\exp|_{N\mathcal{O}(\varepsilon)_{x_0}}$ is an embedding into $\mathcal{H}$,
we have $gx_0=x_0$ and $u_2=gu_1$, that is, $u_1$ and $u_2$ belongs to the same $G_{x_0}$-orbit.
Hence different critical $G_{x_0}$-orbits of $\mathcal{L}_\lambda\circ\exp|_{N\mathcal{O}(\varepsilon)_{x_0}}$
give rise to different critical $G$-orbits of $\mathcal{L}_\lambda$.

Write $\mathcal{L}''(x_0):=d^2(\mathcal{L}\circ\exp|_{B_{T_{x_0}\mathcal{H}}(0,\varepsilon)})(0)$, $\widehat{\mathcal{L}}''(x_0):=d^2(\widehat{\mathcal{L}}\circ\exp|_{B_{T_{x_0}\mathcal{H}}(0,\varepsilon)})(0)$
and $\mathcal{L}''_\lambda(x_0):=d^2(\mathcal{L}_\lambda\circ\exp|_{B_{T_{x_0}\mathcal{H}}(0,\varepsilon)})(0)$
for all $\lambda\in\mathbb{R}$.  Since the orthogonal complementary $N\mathcal{O}_{x_0}$  of $T_{x_0}\mathcal{O}$ in $T_{x_0}\mathcal{H}$
  is an invariant subspace of each $\mathcal{L}''_\lambda(x_0)$, we see that
    $\mathcal{L}''_\lambda(x_0)$ (resp. $\mathcal{L}''(x_0)$, $\widehat{\mathcal{L}}''(x_0)$)
  restricts to a self-adjoint operator from  $N\mathcal{O}_{x_0}$ to itself, denoted by
  $\mathcal{L}''_\lambda(x_0)^\bot$  (resp.  $\mathcal{L}''(x_0)^\bot$,  $\widehat{\mathcal{L}}''(x_0)^\bot$).
 Actually, $\mathcal{L}''_\lambda(x_0)^\bot=d^2(\mathcal{L}_\lambda\circ\exp|_{N\mathcal{O}(\varepsilon)_{x_0}})(0)$ and
$$
\mathcal{L}''(x_0)^\bot=d^2(\mathcal{L}\circ\exp|_{N\mathcal{O}(\varepsilon)_{x_0}})(0),\quad
\widehat{\mathcal{L}}''(x_0)^\bot=d^2(\widehat{\mathcal{L}}\circ\exp|_{N\mathcal{O}(\varepsilon)_{x_0}})(0).
$$
Note that the induced $G_{x_0}$-actions on $T_{x_0}\mathcal{H}$ and $N\mathcal{O}_{x_0}$ are $C^2$
actions via Hilbert space isomorphisms on $T_{x_0}\mathcal{H}$ and $N\mathcal{O}_{x_0}$, respectively.
Applying Corollaries~\ref{cor:Bi.2.4.1},~\ref{cor:Bi.2.4.2} to
$(\mathcal{L}\circ\exp|_{N\mathcal{O}(\varepsilon)_{x_0}}, \widehat{\mathcal{L}}\circ\exp|_{N\mathcal{O}(\varepsilon)_{x_0}},
N\mathcal{O}(\varepsilon)_{x_0})$, respectively,
we get the following improvements of Theorems~5.18 and 5.19 in \cite{Lu8}.

\begin{theorem}\label{th:Bi.3.20.1}
Under Hypothesis~\ref{hyp:Bi.3.19},
suppose that  $\lambda^\ast\in\mathbb{R}$ is an isolated eigenvalue of
\begin{equation}\label{e:Bi.3.12}
\mathcal{L}''(x_0)^\bot v-\lambda\widehat{\mathcal{L}}''(x_0)^\bot v=0,\quad v\in N\mathcal{O}_{x_0},
\end{equation}
and  that $\widehat{\mathcal{L}}''(x_0)^\bot$ is either semi-positive or semi-negative.
 Then $\mathcal{O}$ is a  bifurcation $G$-orbit with parameter $\lambda^\ast$ of  the equation
(\ref{e:Bi.3.11}), and one of the following alternatives occurs:
\begin{description}
\item[(i)] $\mathcal{O}$ is not an isolated critical orbit of $\mathcal{L}_{\lambda^\ast}$.

\item[(ii)]  For every $\lambda\in\mathbb{R}$ near $\lambda^\ast$ there is a critical point
$u_\lambda\notin\mathcal{O}$ of $\mathcal{L}_{\lambda}$ converging to $x_0$ as $\lambda\to\lambda^\ast$.

\item[(iii)]  For any given $G$-invariant neighborhood $\mathcal{N}$ of $\mathcal{O}$ in $\mathcal{H}$,
there is an one-sided  neighborhood $\Lambda$ of $\lambda^\ast$ in $\mathbb{R}$ such that
for any $\lambda\in\Lambda\setminus\{\lambda^\ast\}$,
$\mathcal{L}_{\lambda}$ has  at least two critical $G$-orbit in $\mathcal{N}$  which
are different from $\mathcal{O}$, provided that any nontrivial orbit $\mathcal{O}^\ast$ near $0$ of the action of $G_{x_0}$ on
$$
X:={\rm Ker}(\mathcal{L}''(x_0)^\bot -\lambda^\ast\widehat{\mathcal{L}}''(x_0)^\bot)
$$
has the Euler-Poincar\'e characteristic
$\chi(\mathcal{O}^\ast)\ne 1-(-1)^{\dim X}$. Moreover, for $\lambda\in\Lambda\setminus\{\lambda^\ast\}$, if $\mathcal{L}_{\lambda}$
 has only finitely many critical $G$-orbits in $\mathcal{N}$,
then it has at least two critical $G$-orbit in $\mathcal{N}$  which
are different from $\mathcal{O}$ and have distinct energy,
 provided that  $\dim X>1$ and any nontrivial orbit $\mathcal{O}^\ast$ near $0$
  of the $G_{x_0}$-action on $X$ satisfies one of the following conditions:
  \begin{description}
\item[iii-1)]  $\dim\mathcal{O}^\ast=0$ or  $1\le\dim\mathcal{O}^\ast\le \dim X-2$.
\item[iii-2)]  $1\le \dim \mathcal{O}^\ast=\dim X-1$, either $\mathcal{O}^\ast$ is non-connected or
$\mathcal{O}^\ast$ is connected and $H_r(\mathcal{O}^\ast, \mathbb{Z}_2)\ne H_r(S^{{\dim X}-1}, \mathbb{Z}_2)$ for some $0\le r\le \dim X-1$.
\end{description}
\end{description}
\end{theorem}

\begin{theorem}\label{th:Bi.3.20.2}
Under Hypothesis~\ref{hyp:Bi.3.19},  the conclusions of Theorem~\ref{th:Bi.3.20.1} hold true
if the assumption ``$\widehat{\mathcal{L}}''(x_0)^\bot$ is either semi-positive or semi-negative''
are replaced by the following: \\
{\bf I)}  $\mathcal{L}''(x_0)^\bot$ is invertible.\\
{\bf II)} $\lambda^\ast=\lambda_{k_0}$ is an eigenvalue of (\ref{e:Bi.3.12}) as above.\\
{\bf III)} One of the following two conditions is satisfied:
 \begin{enumerate}
\item[\rm (a)] $\mathcal{L}''(x_0)^\bot$ is positive;
 \item[\rm (b)] each $N\mathcal{O}_{x_0}^k={\rm Ker}(\mathcal{L}''(x_0)^\bot-\lambda_k\widehat{\mathcal{L}}''(x_0)^\bot)$
 with $k\in\mathbb{N}$
 is an invariant subspace of $\mathcal{L}''(x_0)^\bot$  (e.g. these are true if $\mathcal{L}''(x_0)^\bot$
 commutes with $\widehat{\mathcal{L}}''(x_0)^\bot$), and $\mathcal{L}''(x_0)^\bot$
 is either positive definite or negative one on $N\mathcal{O}_{x_0}^{k_0}$.
 \end{enumerate}
\end{theorem}

\noindent{\bf Acknowledgments}\quad
 The author is deeply grateful to
 the anonymous referee for pointing out some questions and improved suggestions.\\




\renewcommand{\refname}{REFERENCES}

\medskip

\begin{tabular}{l}
 School of Mathematical Sciences, Beijing Normal University\\
 Laboratory of Mathematics and Complex Systems, Ministry of Education\\
 Beijing 100875, The People's Republic of China\\
 E-mail address: gclu@bnu.edu.cn\\
\end{tabular}

\end{document}